\documentclass[a4paper,12pt]{amsart}

\usepackage{amsmath,amsfonts,amssymb,amsthm}
\usepackage{calligra, mathrsfs}
\usepackage{mathtools}

\usepackage{subfig}
\usepackage{graphicx}
\usepackage{color}
\usepackage[all]{xy}
\usepackage{url}
\usepackage{setspace}
\usepackage{booktabs}
\usepackage{longtable}
\usepackage{latexsym}
\usepackage{comment}
\usepackage[usenames,dvipsnames]{xcolor}
\usepackage{array}
\newcolumntype{L}[1]{>{\raggedright\let\newline\\\arraybackslash\hspace{0pt}}m{#1}}
\newcolumntype{C}[1]{>{\centering\let\newline\\\arraybackslash\hspace{0pt}}m{#1}}
\newcolumntype{R}[1]{>{\raggedleft\let\newline\\\arraybackslash\hspace{0pt}}m{#1}}

\usepackage{enumitem}
\usepackage{afterpage}
\usepackage{hyperref}

\usepackage{tikz}
\usetikzlibrary{cd}
\usetikzlibrary{patterns}
\usetikzlibrary{shapes}
\usetikzlibrary{decorations.fractals}

\newtheorem{theorem}{Theorem}
\newtheorem{lemma}[theorem]{Lemma}
\newtheorem{proposition}[theorem]{Proposition}
\newtheorem{corollary}[theorem]{Corollary}

\theoremstyle{definition}
\newtheorem{definition}[theorem]{Definition}
\newtheorem{example}[theorem]{Example}

\newtheorem{remark}[theorem]{Remark}

\newtheorem{construction}[theorem]{Construction}

\numberwithin{equation}{section}

\definecolor{cKlaus}{rgb}{0.1,0.55,0.03}
\definecolor{cKlausOK}{rgb}{0.6,0.10,0.33}
\definecolor{intOrange}{rgb}{1.0,.310,.0}
\usepackage[textsize=tiny]{todonotes}

\newcommand{\cN}{\mathcal{N}}

\newcommand{\Cstar}{\mathbb{C}^*}

\newcommand{\bC}{\mathbb{C}}
\newcommand{\bR}{\mathbb{R}}

\newcommand{\bZ}{\mathbb{Z}}

\DeclareMathOperator{\rank}{rk}
\DeclareMathOperator{\tail}{tail}

\DeclareMathOperator{\Hom}{Hom}
\DeclareMathOperator{\sheafhom}{\mathscr{H}\text{\kern -3pt {\calligra\large om}}\,}
\DeclareMathOperator{\sheafext}{\mathscr{E}\text{\kern -3pt {\calligra\large xt}}\;}

\DeclareMathOperator{\orb}{orb}

\DeclareMathOperator{\Coh}{Coh}

\DeclareMathOperator{\Spec}{Spec}

\DeclareMathOperator{\id}{id}

\DeclareMathOperator{\pr}{pr}

\newcommand{\wt}{\widetilde}

\DeclareMathOperator{\vertset}{vert}
\DeclareMathOperator{\cone}{cone}

\newcommand{\idm}{\mathfrak{m}}

\newcommand{\PP}{\mathbb P}

\newcommand{\F}{\mathbb F}

\newcommand{\N}{\mathbb N}
\newcommand{\R}{\mathbb R}
\newcommand{\T}{\mathbb T}
\newcommand{\V}{\mathbb V}
\newcommand{\Z}{\mathbb Z}
\newcommand{\C}{\mathbb C}

\newcommand{\CC}{{\mathcal C}}

\newcommand{\CE}{{\mathcal E}}
\newcommand{\CF}{{\mathcal F}}
\newcommand{\CG}{{\mathcal G}}
\newcommand{\CH}{{\mathcal H}}

\newcommand{\CL}{{\mathcal L}}
\newcommand{\CO}{{\mathcal O}}

\newcommand{\CS}{{\mathcal S}}
\newcommand{\CT}{{\mathcal T}}
\newcommand{\CU}{{\mathcal U}}

\definecolor{intOrange}{rgb}{1.0,.310,.0}

\renewcommand{\iff}{\Leftrightarrow}

\newcommand{\gExt}{\mbox{\rm Ext}}

\renewcommand{\div}{\operatorname{div}}
\newcommand{\largestint}[1]{\lfloor #1 \rfloor}
\newcommand{\rounddown}[1]{\largestint{#1}}

\newcommand{\spann}{\operatorname{span}}

\newcommand{\kst}{\,|\;}

\newcommand{\kss}{\scriptscriptstyle}
\newcommand{\kbb}{{\kss \bullet}}
\newcommand{\ko}{\overline}

\newcommand{\rato}{-\hspace{-0.3em}\to}

\newcommand{\Pol}{\operatorname{Pol}}
\newcommand{\dual}{^{\scriptscriptstyle\vee}}

\newcommand{\normal}{\cN}

\newcommand{\koszulindex}{p}

\newcommand{\rankM}{r}

\newcommand{\toric}{\T\V}  
\newcommand{\gH}{\operatorname{H}}
\newcommand{\tH}{\wt{\operatorname{H}}}
\newcommand{\Pic}{\operatorname{Pic}}
\newcommand{\Cl}{\operatorname{Cl}}

\definecolor{skin}{HTML}{FFECC9}
\definecolor{pumpkin}{HTML}{FEDFA9}
\definecolor{piggy}{HTML}{FFB99D}
\definecolor{fiolet}{HTML}{CD8F9C}
\definecolor{granat}{HTML}{677081}
\definecolor{ciemnyblekit}{HTML}{91A1B8}

\definecolor{oliwkowy}{HTML}{627037}
\definecolor{ciemnazielen}{HTML}{394D2E}
\definecolor{ciemnyfiolet}{HTML}{424444}
\definecolor{mocnyfiolet}{HTML}{717299}
\definecolor{jasnyfiolet}{HTML}{B0ABCC}
\definecolor{bladyfiolet}{HTML}{C9C7DB}
\definecolor{lightblue}{RGB}{135,206,250}
\definecolor{darkblue}{RGB}{0,0,160}
\definecolor{darkgreen}{RGB}{0,160,0}

\makeatletter
\def\thickhline{
	\noalign{\ifnum0=`}\fi\hrule \@height \thickarrayrulewidth \futurelet
	\reserved@a\@xthickhline}
\def\@xthickhline{\ifx\reserved@a\thickhline
	\vskip\doublerulesep
	\vskip-\thickarrayrulewidth
	\fi
	\ifnum0=`{\fi}}
\makeatother

\newlength{\thickarrayrulewidth}
\makeatletter
\newcommand{\doublewidetilde}[1]{{
		\mathpalette\double@widetilde{#1}
}}
\newcommand{\double@widetilde}[2]{
	\sbox\z@{$\m@th#1\widetilde{#2}$}
	\ht\z@=.9\ht\z@
	\widetilde{\box\z@}
}
\makeatother

\begin{document}

\title[Extensions of toric line bundles]
{Extensions of toric line bundles}

\author[K.~Altmann]{Klaus Altmann
}
\address{Institut f\"ur Mathematik,
FU Berlin,
K\"onigin-Luise-Str.~24-26,
14195 Berlin,
Germany}
\email{altmann@math.fu-berlin.de}
\author[A.~Flatt]{Amelie Flatt
}
\address{Institut f\"ur Mathematik,
HU Berlin,
Rudower Chaussee 25,
12489 Berlin,
Germany}
\email{amelie.flatt@hu-berlin.de}
\author[L.~Hille]{Lutz Hille
}
\address{Mathematisches Institut der
Universit\"at M\"unster,
Einsteinstr.~62,
48149 M\"unster,
Germany}
\email{lutz.hille@uni-muenster.de}

\begin{abstract}
For any two nef line bundles $\CL^+ \coloneqq \CO_X(\Delta^+)$ and 
$\CL^- \coloneqq \CO_X(\Delta^-)$ on a toric variety X
represented by lattice polyhedra $\Delta^+$ respectively $\Delta^-$,
we present the universal equivariant extension of 
$\CL^-$ by $\CL^+$
under use of the connected components of the set theoretic difference
$\Delta^-\setminus\Delta^+$.
\end{abstract}

\maketitle
\setcounter{tocdepth}{1}

\section{Introduction}
\label{section:introduction}

\subsection{Spotting Cohomology}
\label{subsection:spottingCohomology}
Consider a projective toric variety $X = \PP(\Delta)$ corresponding to a lattice polytope \mbox{$\Delta \subseteq M_\bR$},  $M_\bR = M \otimes_\bZ \bR$ for a lattice $M$,
over $\C$.
A torus invariant Cartier divisor $D$ on $X$ can be represented
by a pair $(\Delta^+,\Delta^-)$ of lattice polytopes, where both 
polytopes encode nef divisors $D_{\Delta^+}$ and $D_{\Delta^-}$ on $X$ and $D = D_{\Delta^+} - D_{\Delta^-}$.
Denote the associated line bundle by $\CO_X(D) \eqqcolon \CO_X(\Delta^+ - \Delta^-)$.\\
It is well-known that the cohomology groups $\gH^i\big(X,\CO_X(D)\big)$ of the line bundle $\CO_X(D)$ of a torus invariant Cartier divisor $D$ are $M$-graded (compare for example \cite[§9.1]{CoxBook} or \cite[Section 3.5]{fultonToric}):
\begin{equation}
\gH^i\big(X,\CO_X(D)\big) = \bigoplus_{m \in M} \gH^i\big(X,\CO_X(D)\big)_m.
\end{equation}
By \cite[Thm. III.6]{immaculate} 
or \cite{dop} we can describe
their homogeneous component of degree $m\in M$ by
\begin{equation}
\gH^i\big(X,\,\CO_X(\Delta^+-\Delta^-)\big)_m = 
\tH^{i-1}\big(\Delta^- \setminus (\Delta^+-m),\,\C\big),
\end{equation}
where $\tH^{i-1}(Z, \C)$ on the right hand side denotes the reduced singular cohomology of a topological space $Z$ with complex coefficients, $(\Delta^+ - m)$ denotes the polytope $\Delta^+$ shifted by the lattice point $-m \in M$ and $\Delta^- \setminus (\Delta^+ - m)$ denotes the set-theoretic difference of the two polytopes.
Recall the $(-1)$-st reduced singular cohomology:
\begin{equation}
\tH^{-1}(Z,\C)=\left\{
\begin{array}{ll}
\C, &\mbox{if } Z= \emptyset,\\
0,&\mbox{otherwise}.
\end{array}\right.
\end{equation}
The $0$-th reduced cohomology is the quotient $\tH^0(Z, \C) = \gH^0(Z, \bC) \big/ \gH^0(\{\cdot\}, \bC)$, and thus, its dimension is the number of connected components of $Z$ minus 1.

\subsection{The Example {$\F_1$}}
\label{subsection:exampleF1}
We look at the first Hirzebruch surface {$\F_1$} as a first example.
Consider the projections $p_1 \colon\F_1\to\PP^1$ as a ruled surface 
and \mbox{$p_2\colon\F_1\to\PP^2$} as a blow-up. 
We use \mbox{$\CO_{\F_1}(1,0):=p_1^*\,\CO_{\PP^1}(1)$}
and \mbox{$\CO_{\F_1}(0,1):=p_2^*\,\CO_{\PP^2}(1)$}
as a basis for \mbox{$\Pic\F_1=\Z^2$}.
For $i,j\geq 0$ the sheaves $\CO_{\F_1}(i,j)$
correspond to  lattice polytopes $\Delta_{(i,j)}$ in $\bR^2$. For example, for
$\CO_{\F_1}(0,0)$, $\CO_{\F_1}(1,0)$,
$\CO_{\F_1}(0,1)$, $\CO_{\F_1}(0,2)$, 
and the ample $\CO_{\F_1}(1,1)$, the lattice polytopes look as follows:
$$
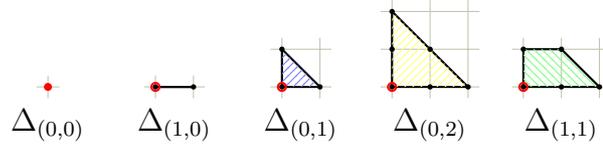
\newcommand{\scaleA}{0.5}
\newcommand{\spaceA}{\hspace*{1em}}
\begin{tikzpicture}[scale=\scaleA]
\draw[color=oliwkowy!40] (-0.3,-0.3) grid (0.3,0.3);
\fill[thick, color=red]
(0,0) circle (3pt);
\draw[thick,  color=black]
(0,-1) node{$\Delta_{(0,0)}$};
\end{tikzpicture}
\spaceA
\begin{tikzpicture}[scale=\scaleA]
\draw[color=oliwkowy!40] (-0.3,-0.3) grid (1.3,0.3);
\draw[thick, color=black]
(0,0) -- (1,0) -- (0,0);
\fill[thick, color=black]
(0,0) circle (2pt) (1,0) circle (2pt);
\draw[thick, color=red]
(0,0) circle (3pt);
\draw[thick,  color=black]
(0.5,-1) node{$\Delta_{(1,0)}$};
\end{tikzpicture}
\spaceA
\begin{tikzpicture}[scale=\scaleA]
\draw[color=oliwkowy!40] (-0.3,-0.3) grid (1.3,1.3);
\draw[thick, color=black]
(0,0) -- (1,0) -- (0,1) -- (0,0);
\fill[pattern color=blue!50, pattern=north east lines]
(0,0) -- (1,0) -- (0,1) -- (0,0);
\fill[thick, color=black]
(0,0) circle (2pt) (1,0) circle (2pt) (0,1) circle (2pt);
\draw[thick, color=red]
(0,0) circle (3pt);
\draw[thick,  color=black]
(0.5,-1) node{$\Delta_{(0,1)}$};
\end{tikzpicture}
\spaceA
\begin{tikzpicture}[scale=\scaleA]
\draw[color=oliwkowy!40] (-0.3,-0.3) grid (2.3,2.3);
\draw[thick, color=black]
(0,0) -- (2,0) -- (0,2) -- (0,0);
\fill[pattern color=yellow!50, pattern=north east lines]
(0,0) -- (2,0) -- (0,2) -- (0,0);
\fill[thick, color=black]
(0,0) circle (2pt) (1,0) circle (2pt) (2,0) circle (2pt)
(0,1) circle (2pt) (1,1) circle (2pt) (0,2) circle (2pt);
\draw[thick, color=red]
(0,0) circle (3pt);
\draw[thick,  color=black]
(1,-1) node{$\Delta_{(0,2)}$};
\end{tikzpicture}
\spaceA
\begin{tikzpicture}[scale=\scaleA]
\draw[color=oliwkowy!40] (-0.3,-0.3) grid (2.3,1.3);
\draw[thick, color=black]
(0,0) -- (2,0) -- (1,1) -- (0,1) -- (0,0);
\fill[pattern color=green!50, pattern=north west lines]
(0,0) -- (2,0) -- (1,1) -- (0,1) -- (0,0);
\fill[thick, color=black]
(0,0) circle (2pt) (1,0) circle (2pt) (2,0) circle (2pt)
(0,1) circle (2pt) (1,1) circle (2pt);
\draw[thick, color=red]
(0,0) circle (3pt);
\draw[thick,  color=black]
(1,-1) node{$\Delta_{(1,1)}$};
\end{tikzpicture}
$$
The red dot indicates the position of the origin in each figure, fixing the exact position of each polytope within the plane. 
We see an example of the result quoted in
subsection~(\ref{subsection:spottingCohomology}).
Consider the polytope $\Delta_{(0,2)}$ shaded in yellow and the polytope $\Delta_{(1,0)} + (0,1)$, that is, $\Delta_{(1,0)}$ shifted by $(0,1) \in \bZ^2$, depicted in \color{orange} orange\color{black}:
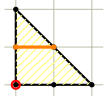
\begin{center}
	\newcommand{\scaleA}{0.5}
	\begin{tikzpicture}[scale=\scaleA]
	\draw[color=oliwkowy!40] (-0.3,-0.3) grid (2.3,2.3);
	\draw[thick, color=black]
	(0,0) -- (2,0) -- (0,2) -- (0,0);
	\fill[pattern color=yellow!50, pattern=north east lines]
	(0,0) -- (2,0) -- (0,2) -- (0,0);
	\fill[thick, color=black]
	(0,0) circle (2pt) (1,0) circle (2pt) (2,0) circle (2pt)
	(0,1) circle (2pt) (1,1) circle (2pt) (0,2) circle (2pt);
	\fill[thick, color=orange]
	(0,1) circle (2pt) (1,1) circle (2pt);
	\draw[thick, color=red]
	(0,0) circle (3pt);
	\draw[very thick, color=orange]
	(0,1) -- (1,1);
	\end{tikzpicture}
\end{center}
The two connected components of the set-theoretic difference $\Delta_{(0,2)} \setminus \big(\Delta_{(1,0)} + (0,1)\big)$ provide a one-dimensional $0$-th reduced singular cohomology
\begin{equation}
\tH^0\big(\Delta_{(0,2)} \setminus (\Delta_{(1,0)} + (0,1))\big),
\end{equation}
and so by \cite{immaculate} and \cite{dop} a one-dimensional piece (in fact, the only one) of 
\begin{equation}
\gH^1\big(\F_1,\;\CO_{\F_1}(\Delta_{(1,0)}-\Delta_{(0,2)})\big)=
\gH^1\big(\F_1,\;\CO_{\F_1}(1,-2)\big),
\end{equation}
sitting in degree $m=-(0,1) \in \bZ^2$.\\[1ex]
In this paper we will take another point of view.
The partition of $\Delta_{(0,2)}$ into two connected components induces the following
``exact sequences of polytopes'':
\begin{equation}\label{equation:exactSequenceOfPolytopes}
\newcommand{\scaleA}{0.5}
\newcommand{\spaceA}{\hspace*{1em}}
0 \hspace{0.5em} \to
\spaceA
\begin{tikzpicture}[scale=\scaleA]
\draw[color=oliwkowy!40] (-0.3,-0.3) grid (1.3,0.3);
\draw[thick, color=black]
(0,0) -- (1,0) -- (0,0);
\fill[thick, color=black]
(0,0) circle (2pt) (1,0) circle (2pt);
\draw[thick, color=red]
(0,0) circle (3pt);
\end{tikzpicture}
\spaceA
\to
\spaceA
\raisebox{-0.5em}{\fbox{$
		\begin{tikzpicture}[scale=\scaleA]
		\draw[color=oliwkowy!40] (-0.3,-0.3) grid (1.3,1.3);
		\draw[thick, color=black]
		(0,0) -- (1,0) -- (0,1) -- (0,0);
		\fill[pattern color=blue!50, pattern=north east lines]
		(0,0) -- (1,0) -- (0,1) -- (0,0);
		\fill[thick, color=black]
		(0,0) circle (2pt) (1,0) circle (2pt) (0,1) circle (2pt);
		\draw[thick, color=red]
		(0,0) circle (3pt);
		\end{tikzpicture}
		\raisebox{0.5em}{$\oplus$}
		\begin{tikzpicture}[scale=\scaleA]
		\draw[color=oliwkowy!40] (-0.3,-0.3) grid (2.3,1.3);
		\draw[thick, color=black]
		(0,0) -- (2,0) -- (1,1) -- (0,1) -- (0,0);
		\fill[pattern color=green!50, pattern=north west lines]
		(0,0) -- (2,0) -- (1,1) -- (0,1) -- (0,0);
		\fill[thick, color=black]
		(0,0) circle (2pt) (1,0) circle (2pt) (2,0) circle (2pt)
		(0,1) circle (2pt) (1,1) circle (2pt);
		\draw[thick, color=red]
		(0,1) circle (3pt);
		\end{tikzpicture}
		$}}
\spaceA
\to
\spaceA
\raisebox{-1em}{\begin{tikzpicture}[scale=\scaleA]
	\draw[color=oliwkowy!40] (-0.3,-0.3) grid (2.3,2.3);
	\draw[thick, color=black]
	(0,0) -- (2,0) -- (0,2) -- (0,0);
	\fill[pattern color=yellow!50, pattern=north east lines]
	(0,0) -- (2,0) -- (0,2) -- (0,0);
	\fill[thick, color=black]
	(0,0) circle (2pt) (1,0) circle (2pt) (2,0) circle (2pt)
	(0,1) circle (2pt) (1,1) circle (2pt) (0,2) circle (2pt);
	\draw[thick, color=red]
	(0,1) circle (3pt);
	\draw[thick, color=orange]
	(0,1) -- (1,1);
	\end{tikzpicture}}
\spaceA
\to \hspace{0.5em} 0.
\end{equation}
The two polytopes in the middle cover the polytope $\Delta_{(0,2)} - (0,1)$ and intersect in the polytope $\Delta_{(1,0)}$, hence they give an inclusion/exclusion sequence of polytopes.
In section~\ref{section:polyhedraToSheaves} we show that this 
corresponds to an exact sequence of sheaves
\begin{equation}\label{equation:exactSequeceInducedBySequenceOfPolytopes}
0 \to \CO_{\F_1}(1,0) \to \CO_{\F_1}(0,1) \oplus \CO_{\F_1}(1,1) \to \CO_{\F_1}(0,2) \to 0.
\end{equation}
We obtain an extension
of $\CO_{\F_1}(0,2)$ by $\CO_{\F_1}(1,0)$, that is, an element of the group 
\begin{align}
\gExt^1(\CO_{\F_1}(0,2), \CO_{\F_1}(1,0))
&= \gH^1\big(\F_1,\;\CO_{\F_1}(1,-2)\big) \\
&= \gH^1\big(\F_1,\;\CO_{\F_1}(\Delta_{(1,0)}-\Delta_{(0,2)}\big),
\end{align}
which we know to be one-dimensional by \cite{immaculate} and \cite{dop}.
We will show that the short exact extension sequence~(\ref{equation:exactSequeceInducedBySequenceOfPolytopes}) induced by the ``exact sequences of polytopes''~(\ref{equation:exactSequenceOfPolytopes}) represents this one-dimensional vector space
\mbox{$\gExt^1\big(\CO_{\F_1}(0,2),\CO_{\F_1}(1,0)\big)$} and, moreover, that this
concept works in general.

\subsection{Acknowledgements}
\label{subsection:acknowledgements}
We would like to thank Christian Haase for helpful discussions throughout the work on this paper, in particular for the triangulation argument in the proof of Proposition 13.
{Moreover, we thank the anonymous referee for the careful reading 
and many valuable hints and remarks.
The second author was supported by the BMS ("Berlin Mathematical School")
and the third one by the
DFG and Exzellenzcluster "Mathematik M\"unster".
\footnote{Gef\"ordert durch die Deutsche Forschungsgemeinschaft (DFG) im
Rahmen der Exzellenzstrategie des Bundes und der Länder EXC 2044
–390685587, Mathematik M\"unster: Dynamik–Geometrie–Struktur}
}

\section{Toric Geometry}
\label{section:toricGeometry}
We introduce some basics of toric geometry central to this paper. Readers not familiar with toric geometry can take a look at one of the numerous introductory texts, for example \cite{CoxBook}, \cite{fultonToric}, \cite{Danilov_1978}, or \cite{Demazure1970}.\\[1ex]
Let $M \cong \bZ^\rankM$ be a lattice and $N = \Hom_\bZ(M, \bZ) \cong \bZ^\rankM$ its dual lattice. There is a natural pairing
$\langle \cdot , \cdot \rangle \colon M \times N \to \bZ.$
We consider the algebraic torus \mbox{$T = \Spec \C[M]$}. 
The isomorphism $M \cong \bZ^\rankM$ induces an isomorphism $T \cong (\Cstar)^\rankM$. 
The lattice $M$ can be recovered as the \emph{character lattice} $\Hom(T, \Cstar)$. 
We denote the character of \mbox{$M \ni m \mapsto (a_1, \dots, a_r) \in \bZ^\rankM$} by 
\mbox{$\chi^m \colon T \to \Cstar, \, (t_1, \dots, t_r) \mapsto t_1^{a_1}\cdots t_r^{a_r}.$}
The dual lattice $N$ corresponds to the \emph{group of $1$-parameter subgroups} $\Hom(\Cstar, T)$. Here $N \ni n \mapsto (b_1, \dots, b_r) \in \bZ^\rankM$ corresponds to 
\mbox{$\lambda^n  \colon \Cstar \to T, \, t \mapsto (t^{b_1}, \dots, t^{b_r}).$}
A \emph{toric variety} is an irreducible variety containing an algebraic torus $T \cong (\Cstar)^\rankM$ as an open dense subset, such that the action of the torus on itself by multiplication extends to an algebraic action on the whole variety \cite[Def. 3.1.1]{CoxBook}. 
We sketch how \emph{normal toric varieties} can be constructed from cones and fans in $N_\bR = N \otimes_\bZ \bR$.\\[1ex]
By a \emph{cone} in $N_\bR$ we mean a convex subset 
\mbox{$\sigma =  \cone(S) = \{ \sum_{i = 1}^{k} \lambda_i v_i \kst \lambda_i \ge 0\}$}
generated by a finite set $S = \{v_1, \dots, v_k\} \subseteq N$.
The \emph{dual cone} to $\sigma$ in $M_\bR$ is
\mbox{$\sigma^\vee = \{u \in M_\bR \kst \langle u, v \rangle \ge 0 \text{ for all } v \in \sigma\}.$}
A cone $\sigma$ is \emph{pointed} if $\sigma \cap (-\sigma) = \{0\}$. We write $\tau \preceq \sigma$ whenever $\tau$ is a \emph{face} of $\sigma$.
A pointed cone $\sigma \subseteq N_\bR$ leads to an \emph{affine toric variety}
$\toric(\sigma) \coloneqq \Spec \C[\sigma^\vee \cap M]$.
The inclusion of a face $\tau \preceq \sigma$ induces an open embedding $\toric(\tau) \hookrightarrow \toric(\sigma)$. In particular, the inclusion of the origin induces an open embedding of the torus \mbox{$T = \Spec \C[M] \hookrightarrow \toric(\sigma)$}.\\[1ex]
A \emph{fan} $\Sigma$ in $N_\bR$ is a finite collection of pointed cones that is closed under taking faces and such that the intersection of two cones is a face of each.
The affine toric varieties $U_\sigma \coloneqq \toric(\sigma)$ associated to the cones $\sigma$ in a fan $\Sigma$ glue together to a \emph{normal separated} toric variety $\toric(\Sigma)$ with open affine charts the $U_\sigma$ \cite[Thm. 3.1.5]{CoxBook}.\\[1ex]
A \emph{polyhedron} $\Delta$ in $M_\bR$ is the intersection of finitely many closed half spaces
\begin{equation}
\Delta = \{m \in M_\bR \kst \langle m, v_i \rangle \ge -\lambda_i, \, i = 1, \dots, s \} = \bigcap_{i = 1}^s H_i^+,
\end{equation} 
where $H_i^+ = \{m \in M_\bR \kst \langle m, v_i \rangle \ge -\lambda_i\}$ for some inward pointing normal vector \mbox{$v_i \in N_\bR$} and some scalar $\lambda_i \in \bR$ for $i \in \{1, \dots, s\}$. 
A compact polyhedron is called a \emph{polytope}. A polyhedron $\Delta$ can be written as the \emph{Minkowski sum} $\Delta = \nabla + \delta$ of a polytope $\nabla$ and its \emph{tail cone} 
\mbox{$\delta = \tail(\Delta) = \{m \in M_\bR \kst u + m \in \Delta \text{ for all } u \in \Delta\} \subseteq M_\bR.$}
To a full-dimensional lattice polyhedron $\Delta$ we associate its \emph{inner normal fan} 
\begin{equation}
\normal(\Delta) \coloneqq \{\tau \kst \tau \preceq \sigma_m, \, m \in \vertset(\Delta)\} \text{ with } |\normal(\Delta)|= \tail(\Delta)^\vee,
\end{equation}
whose maximal cones $\sigma_m$ are given by
$\sigma_m \coloneqq (\cone(\Delta - m))^\vee = \bR_{\ge 0} \cdot (\Delta - m)^\vee \subseteq N_\bR.$
This yields a semiprojective toric variety $\PP(\Delta) \coloneqq \toric(\normal(\Delta))$. 
We will assume all polyhedra to have at least one vertex. The normal fan of such a polyhedron will have convex support of full dimension.

\subsection{Divisors on Toric Varieties}
\label{subsection:divisorsOnToricVarieties}
Let $\Sigma$ be a fan in $N_\bR$ with convex support of full dimension $\rankM = \dim N_\bR$, for example, the normal fan of a full-dimensional lattice polyhedron in $M_\bR$. Let $X \coloneqq \toric(\Sigma)$ be the toric variety given by $\Sigma$.\\[1ex]
Every lattice polyhedron $\Delta \subseteq M_\bR$ with tail cone $|\Sigma|^\vee$ and whose normal fan $\normal(\Delta)$ is refined by $\Sigma$ gives rise to a nef Cartier divisor $D_\Delta$ on $X$ by the following construction: 
If $\Sigma$ refines $\normal(\Delta)$ then the function 
$|\Sigma| \to \bR, \, v \mapsto \min \langle \Delta, v \rangle$ 
is linear on the cones of $\Sigma$. 
So for each $\sigma \in \Sigma$ there is some $m_\sigma \in \Delta \cap M$ such that $\min \langle \Delta, v \rangle = \langle m_\sigma, v \rangle$ for each $v \in \sigma$. 
The lattice point $m_\sigma$ is a vertex of $\Delta$ and if $\sigma \in \Sigma(\rankM)$ is a maximal cone, then $m_\sigma$ is uniquely determined. Note that
$\Delta + \sigma^\vee = m_\sigma + \sigma^\vee.$
Let $D_\Delta$ be the Cartier divisor with Cartier data $\{m_\sigma\}_{\sigma \in \Sigma}$, that is, locally on $U_\sigma \subseteq X$ given by $D_{|U_\sigma} = \div(\chi^{-m_\sigma}) _{|U_\sigma}$.
Then for the associated line bundle $\CO_X(\Delta) \coloneqq \CO_X(D_\Delta)$:
\begin{equation}
\Gamma(U_\sigma, \CO_{X}(\Delta)) = \bigoplus_{m \in (\Delta + \sigma^\vee) \cap M} \C \cdot \chi^m = \chi^{m_\sigma} \cdot \C[\sigma^\vee \cap M] = \C[m_\sigma + (\sigma^\vee \cap M)],
\end{equation}
 so the vertex $m_\sigma$ of $\Delta$ encodes the local sections of the line bundle $\CO_{X}(\Delta)$ over the affine open $U_\sigma$. The line bundle $\CO_X(\Delta)$ is \emph{ample} if and only if $\normal(\Delta) = \Sigma$.\\[1ex]
We denote the set of lattice polyhedra with prescribed tail cone $\delta$ by $\Pol_\delta^+$
 and the set of lattice polyhedra \emph{compatible} with $\Sigma$, that is, with tail cone $|\Sigma|^\vee$ and whose normal fan is refined by $\Sigma$, by $\Pol^+(\Sigma)$. 
 The set $\Pol_\delta^+$ forms a semigroup with respect to Minkowski addition. For a fan $\Sigma$ the set $\Pol^+(\Sigma)$ forms a finitely generated subsemigroup. The union over the $\Pol^+(\Sigma)$ with $|\Sigma|^\vee = \delta$ is the semigroup $\Pol_\delta^+$.
These semigroups are cancellative,
{%
which is caused by the presence of the prescribed common tail cone $\delta$:
It makes sure that a polyhedron $\Delta\in\Pol_\delta^+$
is uniquely determined by the values of
$\min\langle\Delta,v\rangle$ with $v$ running through $\delta\dual$.
Hence, these semigroups embed%
}
into their respective Grothendieck groups of formal differences:
 \begin{align}
 &\Pol_\delta^+ \hookrightarrow \Pol_\delta = \{\Delta^+ - \Delta^- \kst \Delta^+, \Delta^- \in \Pol_\delta^+\},\\
 &\Pol^+(\Sigma) \hookrightarrow \Pol(\Sigma) = \{\Delta^+ - \Delta^- \kst \Delta^+, \Delta^- \in \Pol^+(\Sigma)\}.
 \end{align}
 On a quasiprojective toric variety $X = \toric(\Sigma)$, every Cartier divisor $D$ can be written (non-uniquely) as a difference
 $D = D^+ - D^-$
 with both $D^+$ and $D^-$ nef Cartier divisors \cite[Thm. 6.3.22]{CoxBook}. 
 In particular every Cartier divisor on $X$ can be represented by a pair of lattice polyhedra $(\Delta^+, \Delta^-)$ compatible with the fan $\Sigma$, that is, an element of the Grothendieck group $\Pol(\Sigma)$.
 
\section{From Complexes of Polyhedra To Complexes of Sheaves}
\label{section:polyhedraToSheaves}
\subsection{The Koszul Complex of Polyhedra}
\label{subsection:koszulComplexOfPolyhedra}
In this section we construct ``exact sequences of polyhedra" that induce exact sequences of split vector bundles on a toric variety $X = \toric(\Sigma)$ over $\C$. We start out on the polyhedral side and let $\Sigma$ be a fan in $N_\bR$ with convex support of full dimension. 
Let $\Pol^+(\Sigma)$ denote the set of lattice polyhedra in $M_\bR$ compatible with $\Sigma$.
Also include the empty set in $\Pol^+(\Sigma)$. 
\begin{definition}
	\label{def:SigmaFamily}
	A \emph{$\Sigma$-family of polyhedra} is a finite set
	$
	S = \{ \nabla_i \kst i \in I \}
	$
	of lattice polyhedra $\nabla_i \in \Pol^+(\Sigma)$ satisfying the following conditions:
	\begin{enumerate}
		\item $\bigcup_{i \in I} \nabla_i\eqqcolon\nabla \in \Pol^+(\Sigma)$ is a lattice polyhedron compatible with $\Sigma$;
		\item all intersections $\nabla_{I'}\coloneqq\bigcap_{i\in I'}\nabla_i$ 
		with $\emptyset\neq I'\subseteq I$
		are either empty or compatible with $\Sigma$, 
		that is, $\nabla_{I'} \in \Pol^+(\Sigma)$.
		Furthermore, set $\nabla_\emptyset \coloneqq \nabla$.
	\end{enumerate}
\end{definition}
We consider two categories and a functor associated to a $\Sigma$-family between them.
\begin{definition}
	\label{def:functorToSigmaFamily}
	Let $2^I$ be the \emph{poset category} associated to the power set of the finite set $I$, that is, objects are subsets of $I$ and there exists a unique morphism from $I'$ to $I''$ whenever $I' \subseteq I''$. 
	Let $\Pol^+(\Sigma)$ also denote the \emph{category of lattice polyhedra compatible with the fan} $\Sigma$, that is, objects are compatible lattice polyhedra as defined above or the empty set and for two polyhedra $\Delta_1$ and $\Delta_2$ in $\Pol^+(\Sigma)$ we define
	\begin{equation}
	\Hom_{\Pol^+(\Sigma)}(\Delta_1, \Delta_2) = 
	\begin{cases}
	\{\text{inclusion}\} & \,\text{ if } \Delta_1 \subseteq \Delta_2, \\
	\emptyset & \, \text{ if } \Delta_1 \nsubseteq \Delta_2.
	\end{cases}
	\end{equation}
	Given an $\Sigma$-family of polyhedra $S$ we can define a contravariant functor
	\begin{equation}
	F_S\colon2^I \to  \Pol^+(\Sigma) \hspace{1em}
	I' \mapsto  \nabla_{I'} = \bigcap_{i \in I'} \nabla_i,\hspace{1em}
	(I' \subseteq I'') \mapsto (\nabla_{I''} \hookrightarrow \nabla_{I'}).
	\end{equation}
\end{definition}
In general, given any contravariant functor $F\colon2^I\to \Pol^+(\Sigma)$, we define a subcomplex of the Koszul complex $\bigwedge^\kbb\C^I$ as follows.
\begin{definition}
	\label{def:koszulComplexToSigmaFamily}
	For a contravariant functor $F \colon 2^I \to \Pol^+(\Sigma)$ and $\koszulindex \in \N$ let
	\begin{equation}
	\CC^F_\koszulindex:=\spann \{e_{I'}\kst \#I'=\koszulindex\mbox{ and }F(I')\neq\emptyset\}
	\subseteq\bigwedge^\koszulindex \C^I.
	\end{equation}
	Here $e_{I'}\coloneqq \underset{i \in I'}{\bigwedge} e_i$, where $\{e_i\kst i\in I\}$,
	denotes the canonical basis of $\C^I$. \\[1ex]
	The map $d\colon\CC^F_{\koszulindex+1}\to\CC^F_\koszulindex$ is defined as for the Koszul complex $\bigwedge^\kbb\C^I$. 
	For a fixed total order on $I$, say $I =\{1 < \dots < \#I\}$, and $I' = \{i_0 < \dots< i_\koszulindex\} \subseteq I$, we have
	\begin{equation}
	d(e_{I'}) = \sum_{j = 0}^\koszulindex (-1)^j e_{I' \setminus \{i_j\}} = \sum_{i\in I'} (-1)^{|i|} e_{I'\setminus \{i\}},
	\end{equation}
	where $|i|$ refers to the index $j$ of $i = i_j$ in $I'$.\\[1ex]
Recall that the polyhedra $F(I') \in \Pol^+(\Sigma)$ are contained in $M_\bR$.
For each $m\in M_\R$ we define the \emph{evaluation subcomplex} (``at $m$'') as
\begin{equation}
\CC^F_\koszulindex(m):=\spann \{e_{I'}\kst \#I'=\koszulindex\mbox{ and }m\in F(I')\}
\subseteq \CC^F_\koszulindex \subseteq\bigwedge^\koszulindex \C^I
\end{equation}
with the restrictions $d_{|\CC_{p+1}^F(m)} \colon \CC_{p+1}^F(m) \to \CC_p^F(m)$ as boundary maps. 
\end{definition}
\begin{lemma}
	\label{lemma:KoszulSubcomplexExact}
	For a $\Sigma$-family $S$ set $F \coloneqq F_S$. 
	\begin{enumerate}
		\item[(i)] If $m\in M_\R\setminus \nabla$, then
		$\CC^F_\koszulindex(m)=0$ for all $\koszulindex\geq 0$. 
		In particular, $\CC^F_\kbb(m)$ is exact.
		\item[(ii)] If $m\in\nabla\subseteq M_\R$, then $\CC^F_\kbb(m)$ is still exact.
		\item[(iii)] The complex $\CC_\kbb^F = \sum_{m \in M} \CC_\kbb^F(m)$ is exact. 
	\end{enumerate}
\end{lemma}

\begin{proof}
	Statement (i) is clear because $F(I') \subseteq \nabla$ for all $I' \subseteq I$.\\
	For statement (ii) set $I_m:=\{i\in I\kst m\in F(i) \coloneqq F(\{i\}) =\nabla_i\}$.
	Then for $I' \subseteq I$ we have $I' \subseteq I_m$ if and only if $m \in \bigcap_{i \in I'} F(i) = F(I')$.
	Therefore $\CC^F_\kbb(m)$ equals $\bigwedge^\kbb\C^{I_m}$, which is exact \cite[Cor. 4.5.5]{weibel1995introduction}.\\[1ex]
	For statement (iii) first note that we have
	$\CC_\kbb^F = \sum_{m \in M_\R} \CC_\kbb^F(m) = \sum_{m \in \nabla \cap M} \CC_\kbb^F(m).$
	By (i) and (ii) each of the summands $\CC_\kbb^F(m)$ is exact. For any subset $M' \subseteq M_\bR$ the intersection of the evaluation subcomplexes for $m \in M'$ is
	$\bigcap_{m \in M'} \CC_\kbb^F(m) = \bigwedge^\kbb\C^{I_{M'}}$,
	where $I_{M'} \coloneqq \{i \in I \kst M' \subseteq F(i) = \nabla_i\}$.
	Furthermore, we have for $m_1,m_2,m_3\in M_\bR$:
	$$
	(\CC_\kbb^F(m_1) + \CC_\kbb^F(m_2)) \cap \CC_\kbb^F(m_3) = (\CC_\kbb^F(m_1) \cap \CC_\kbb^F(m_3)) + (\CC_\kbb^F(m_2) \cap \CC_\kbb^F(m_3)),
	$$
	because the $\CC_\koszulindex^F(m_i)$ are all subspaces of $\CC_\koszulindex^F$ spanned by a subset of the prescribed basis $\{e_{I'} \kst I' \subseteq I, \#I' = \koszulindex , \,F(I') \ne \emptyset \}$ of $\CC_\koszulindex^F$. 
{\\[1ex] Now, we switch to a
slightly more general setup. Assume that $\CC_\kbb^i\subseteq \CC_\kbb^F$ 
($i=1,\ldots,k$) are complexes such that 
all their mutual intersections are exact and
$$
\textstyle
(\sum_{i\in I} \CC_\kbb^i) \cap \CC_\kbb^j = \sum_{i\in I}
(\CC_\kbb^i\cap \CC_\kbb^j)
$$
for any $j$ and $I\subseteq\{1,\ldots,k\}$. 
Then $\sum_{i=1, \dots, k}\CC_\kbb^i$ is exact, too.
The proof uses induction by $k$ and exploits the short exact sequence
$$
0 \to \big(\sum_{i=1, \dots, k-1} \CC_\kbb^i\big) \cap \CC_\kbb^k  \to
\big(\sum_{i=1, \dots, k-1} \CC_\kbb^i\big) \oplus \CC_\kbb^k \to \sum_{i=1, \dots, k}
\CC_\kbb^i \to 0.
$$
The generalization of the setup is needed to ensure 
that the induction hypothesis
implies that, besides the central term, the left most complex is exact, too. 
\\[1ex]
Finally, we apply the previous claim to the complexes
$\CC_\kbb^i:= \CC^F_\kbb(m)$ where $m\in\nabla \cap M$ replaces
$i=1,\ldots,k$.
}%
Note that even though the set $\nabla \cap M$ may be infinite, the induction ends after finitely many steps since the Koszul complex $\bigwedge^\kbb\C^I$ is finite-dimensional and all $\CC^F_\kbb(m) \subseteq \CC^F_\kbb \subseteq \bigwedge^\kbb\C^I$.
\end{proof}
\begin{example}
	We revisit an example already encountered in the Introduction~(\ref{subsection:exampleF1}).\\
	Let $\Sigma$ be the fan of the first Hirzebruch surface $\F_1 = \PP(\Delta) = \toric(\Sigma)$.
	\newcommand{\scaleA}{0.5}
	\newcommand{\spaceA}{\hspace*{2em}}
	$$
	\begin{tikzpicture}[scale=\scaleA]
	\draw[color=oliwkowy!40] (-0.3,-0.3) grid (2.3,1.3);
	\draw[thick, color=black]
	(0,0) -- (2,0) -- (1,1) -- (0,1) -- (0,0);
	\fill[pattern color=green!50, pattern=north west lines]
	(0,0) -- (2,0) -- (1,1) -- (0,1) -- (0,0);
	\fill[thick, color=black]
	(0,0) circle (2pt) (1,0) circle (2pt) (2,0) circle (2pt)
	(0,1) circle (2pt) (1,1) circle (2pt);
	\draw[thick, color=red]
	(0,0) circle (3pt);
	\draw[thick,  color=black]
	(1,-1) node{$\Delta$};
	\end{tikzpicture}
	\spaceA
	\raisebox{-0.25em}{
	\begin{tikzpicture}[scale=\scaleA]
	\draw[color=oliwkowy!40] (-1.3,-1.3) grid (1.3,1.3);
	\draw[thick, color=red]
	(0,0) circle (3pt);
	\draw[thick,  color=black, ->]
	(0,0) -- (0, -1.3);
	\draw[thick,  color=black, ->]
	(0,0) -- (1.3,0);
	\draw[thick,  color=black, ->]
	(0,0) -- (0,1.3);
	\draw[thick,  color=black, ->]
	(0,0) -- (-1.3,-1.3);
	\draw[thick,  color=black]
	(0.0,-2) node{$\Sigma = \normal(\Delta)$};
	\fill[pattern color=black!30, pattern=north east lines]
	(0,0) -- (1,0) -- (1,-1) -- (0, -1);
	\fill[pattern color=black!30, pattern=north west lines]
	(0,0) -- (-1,-1) -- (0, -1);
	\fill[pattern color=black!30, pattern=north east lines]
	(0,0) -- (-1,-1) -- (-1, 1) -- (0, 1);
	\fill[pattern color=black!30, pattern=north west lines]
	(0,0) -- (1,0) -- (1,1) -- (0, 1);
	\end{tikzpicture}}
	$$
	The fan is complete, so compatible polyhedra will be polytopes.
	Define the $\Sigma$-family
	$$
	S \coloneqq \Bigg\{\raisebox{-1em}{
		\begin{tikzpicture}[scale=\scaleA]
		\draw[color=oliwkowy!40] (-0.3,-0.3) grid (2.3,1.3);
		\draw[thick, color=black]
		(0,0) -- (2,0) -- (1,1) -- (0,1) -- (0,0);
		\fill[pattern color=green!50, pattern=north west lines]
		(0,0) -- (2,0) -- (1,1) -- (0,1) -- (0,0);
		\fill[thick, color=black]
		(0,0) circle (2pt) (1,0) circle (2pt) (2,0) circle (2pt)
		(0,1) circle (2pt) (1,1) circle (2pt);
		\draw[thick, color=red]
		(0,1) circle (3pt);
		\draw[thick,  color=black]
		(1,2) node{$\nabla_0$};
		\end{tikzpicture}},
	\spaceA
	\raisebox{-1em}{\begin{tikzpicture}[scale=\scaleA]
		\draw[color=oliwkowy!40] (-0.3,-0.3) grid (1.3,1.3);
		\draw[thick, color=black]
		(0,0) -- (1,0) -- (0,1) -- (0,0);
		\fill[pattern color=blue!50, pattern=north east lines]
		(0,0) -- (1,0) -- (0,1) -- (0,0);
		\fill[thick, color=black]
		(0,0) circle (2pt) (1,0) circle (2pt) (0,1) circle (2pt);
		\draw[thick, color=red]
		(0,0) circle (3pt);
		\draw[thick,  color=black]
		(0.5,2) node{$\nabla_1$};
		\end{tikzpicture}}
	\Bigg\} \text{ and set } 
	\nabla \coloneqq \nabla_0 \cup \nabla_1 = 
	\raisebox{-1em}{\begin{tikzpicture}[scale=\scaleA]
		\draw[color=oliwkowy!40] (-0.3,-0.3) grid (2.3,2.3);
		\draw[thick, color=black]
		(0,0) -- (2,0) -- (1,1) -- (0,2) -- (0,0);
		\fill[pattern color=yellow!50, pattern=north west lines]
		(0,0) -- (2,0) -- (1,1) -- (0,2) -- (0,0);
		\fill[thick, color=black]
		(0,0) circle (2pt) (1,0) circle (2pt) (2,0) circle (2pt)
		(0,1) circle (2pt) (1,1) circle (2pt) (0,2) circle (2pt);
		\draw[thick, color=red]
		(0,1) circle (3pt);
		\end{tikzpicture}} \in \Pol^+(\Sigma).
	$$
	The complex $\CC_\kbb^F$ is
	$0 \to \bC \cdot (e_0 \wedge e_1) \to \bC\cdot e_0 \oplus \bC \cdot e_1 \to \bC \cdot e_\emptyset \to 0.$
	The evaluation subcomplex $\CC^F_\kbb(m)$ for $m \in M_\R$ is the subcomplex indicating whether
	$m \in \nabla \text{ for } \CC^F_0(m)$,
 	$m \in \nabla_{0}$ respectively $m \in \nabla_{1}$ for $\CC^F_1(m)$, 
	and $m \in \nabla_0 \cap \nabla_1 = 
	\raisebox{-0.35em}{\begin{tikzpicture}[scale=\scaleA]
		\draw[color=oliwkowy!40] (-0.5,-0.5) grid (1.5,0.5);
		\draw[thick, color=black]
		(0,0) -- (1,0);
		\fill[thick, color=black]
		(0,0) circle (2pt) (1,0) circle (2pt) ;
		\draw[thick, color=red]
		(0,0) circle (3pt);
		\end{tikzpicture}}$ for $\CC^F_2(m)$.
\end{example}

There is a particularly nice case in which exactness of the evaluation subcomplexes for all lattice points $m \in M$ is equivalent to exactness for all $m \in M_\bR$.
\begin{lemma}
	\label{lemma:exConeLatt}
	Assume that $\Sigma$ consists of a single cone $\sigma$ and its faces, with $\sigma^\vee$ smooth and full-dimensional. Then for any contravariant functor \mbox{$F \colon 2^I \to \Pol^+(\Sigma)$}, the complexes
	$\CC^F_\kbb(m)$ are exact for all $m\in M$ if and only they are exact for all
	$m\in M_\bR$.
\end{lemma}
\begin{remark}
	This situation occurs naturally when we consider smooth affine toric varieties $\toric(\sigma)$ or an affine open subset $U_\sigma \subseteq \toric(\Sigma)$ for a cone $\sigma \in \Sigma$. In the latter case the polyhedra $F(I')$ are changed to $F(I') + \sigma^\vee$ when considering the affine open $U_\sigma$ (see section~(\ref{subsection:localizationOfKoszulComplex})). In the case of a functor $F_S$ associated to a $\Sigma$-family $S$ the $\nabla_i$ become $\nabla_i + \sigma\dual$ and the $\nabla_{I'}$ for $I' \subseteq I$ become $\nabla_{I'} + \sigma\dual$.
\end{remark}

\begin{proof}[Proof of the Lemma]
	Because $\sigma^\vee$ is smooth, we may choose coordinates such that $M=\bZ^\rankM$ and $\sigma^\vee=\bR^\rankM_{\geq 0}$.
	{%
	For $F(I') \in \Pol^+(\Sigma)$ we either have $F(I')=\emptyset$, in which case $m\in F(I') \iff \rounddown{m}\in F(I')$ is clear, or $F(I')$ has only one vertex and%
	}
	can be written as 
	$
	F(I') = r_{I'} + \R^\rankM_{\geq 0} \text{ for some } r_{I'} \in \Z^\rankM.
	$
	We obtain for every $m\in\R^\rankM$ that
	$$
	m\in F(I') \iff m\geq r_{I'} \iff
	\rounddown{m}\geq r_{I'}\iff \rounddown{m}\in F(I'),
	$$
	where $\rounddown{m}$ and the relation $\geq$ are meant componentwise.
	Hence, $\CC_\kbb^F(m) = \CC_\kbb^F(\rounddown{m})$, proving the claim.
\end{proof}

\begin{example}
	We give a counterexample for the global situation, that is, for $\Sigma$ consisting of more than a cone $\sigma$.
	Let $\Sigma = \normal(\Delta)$ be the fan of $\PP^1 = \PP(\Delta) = \toric(\Sigma)$:
	\newcommand{\scaleA}{0.8}
	\newcommand{\spaceA}{\hspace*{1em}}
	\begin{center}
		$\Delta = 
		\raisebox{-0.35em}{\begin{tikzpicture}[scale=\scaleA]
		\draw[color=oliwkowy!40] (-0.3,-0.3) grid (1.3,0.3);
		\draw[thick, color=black]
		(0,0) -- (1,0);
		\fill[thick, color=black]
		(0,0) circle (2pt) (1,0) circle (2pt);
		\draw[thick, color=red]
		(0,0) circle (3pt);
		\end{tikzpicture}}$ in $M_\bR = \bR$,
		\spaceA
		$\normal(\Delta) = 
		\raisebox{-0.35em}{\begin{tikzpicture}[scale=\scaleA]
		\draw[color=oliwkowy!40] (-1.3,-0.3) grid (1.3,0.3);
		\draw[thick, color=red]
		(0,0) circle (3pt);
		\draw[thick, ->,  color=black]
		(0,0) -- (1.3,0);
		\draw[thick, ->,  color=black]
		(0,0) -- (-1.3,0);
		\end{tikzpicture}}$ in $N_\bR = \bR$.
	\end{center}
	Let $I \coloneqq \{0,1\}$ and define the functor $F \colon 2^I \to \Pol^+(\Sigma)$ as follows
	\begin{align*}
	\emptyset \mapsto
	\raisebox{-0.35em}{\begin{tikzpicture}[scale=\scaleA]
		\draw[color=oliwkowy!40] (-0.3,-0.3) grid (1.3,0.3);
		\draw[thick, color=black]
		(0,0) -- (1,0);
		\fill[thick, color=black]
		(0,0) circle (2pt) (1,0) circle (2pt);
		\draw[thick, color=red]
		(0,0) circle (3pt);
		\end{tikzpicture}}, \spaceA
	\{0\} \mapsto
	\raisebox{-0.35em}{\begin{tikzpicture}[scale=\scaleA]
		\draw[color=oliwkowy!40] (-0.3,-0.3) grid (1.3,0.3);
		(0,0) -- (1,0);
		\fill[thick, color=black]
		(0,0) circle (2pt);
		\draw[thick, color=red]
		(0,0) circle (3pt);
		\end{tikzpicture}}, \spaceA
	\{1\} \mapsto
	\raisebox{-0.35em}{\begin{tikzpicture}[scale=\scaleA]
		\draw[color=oliwkowy!40] (-0.3,-0.3) grid (1.3,0.3);
		\fill[thick, color=black]
		(1,0) circle (2pt);
		\draw[thick, color=red]
		(0,0) circle (3pt);
		\end{tikzpicture}}, \spaceA
	\{0,1\} \mapsto 
	\raisebox{-0.35em}{\begin{tikzpicture}[scale=\scaleA]
		\draw[color=oliwkowy!40] (-0.3,-0.3) grid (1.3,0.3);
		\draw[thick, color=red]
		(0,0) circle (3pt);
		\end{tikzpicture}} = \emptyset.
	\end{align*}
	The only non-trivial evaluation subcomplexes at lattice points  are at $m = 0$, $m = 1$:
	\begin{align*}
	\CC_\kbb^F(0)\colon & 0 \to \C \cdot e_{\{0\}} \xrightarrow{\cong} \C \cdot e_\emptyset \to 0 \\
	\CC_\kbb^F(1)\colon & 0 \to \C \cdot e_{\{1\}} \xrightarrow{\cong} \C \cdot e_\emptyset \to 0
	\end{align*}
	These are exact. But at $m = \frac{1}{2} \in M_\bR$ we have the non-exact evaluation subcomplex
	$$
	\CC_\kbb^F(\frac{1}{2})\colon \spaceA 0 \to 0 \to \C \cdot e_\emptyset \to 0.
	$$
\end{example}

\subsection{Localization of the Koszul Complex}
\label{subsection:localizationOfKoszulComplex}
We use the notation of section (\ref{subsection:koszulComplexOfPolyhedra}).
In section (\ref{section:displayingExt}) we will think about the polyhedra $\nabla_{I'}$ as representing some nef line
bundles on some toric variety $\PP(\Delta)=\toric(\normal(\Delta))$ over $\C$.
Then the complexes constructed in subsection (\ref{subsection:koszulComplexOfPolyhedra}) describe
maps between the global sections of them, and thus among the sheaves
themselves. To understand the local behavior of these complexes, we will
look at the affine charts $\toric(\sigma)\subseteq\PP(\Delta)$ for cones
$\sigma\in \normal(\Delta)$. 

\begin{definition}
	\label{def:localizedFunctor}
	For a contravariant functor $F \colon 2^I \to \Pol^+(\Sigma)$, 
where $|\Sigma|$ is convex and full-dimensional in $N_\bR$, and a cone $\sigma \in \Sigma$ define the functor 
	\begin{equation}
	F^\sigma \colon 2^I \to \Pol^+(\sigma), \, 
                        I' \mapsto F(I') + \sigma\dual.
	\end{equation}
	Because $F(I') \in \Pol^+(\Sigma)$ has tail cone $|\Sigma|\dual$ {or is empty} and $\sigma \in \Sigma$ implies $|\Sigma|\dual \subseteq \sigma\dual$, the image $F(I') +\sigma\dual$ has tail cone \mbox{$|\Sigma|\dual + \sigma\dual = \sigma\dual$} {or is empty}. If $\sigma\dual$ is full-dimensional in $M_\bR$, then the normal fan of $F(I') + \sigma\dual$ consists of just $\sigma$ with its faces, that is, $F(I') + \sigma\dual \in \Pol^+(\sigma)$.
\end{definition}

\begin{proposition}
	\label{prop-stayExact}
	For a polyhedron $\Delta$ and a finite set $I$,
	let \mbox{$F \colon 2^I \to \Pol^+(\normal(\Delta))$} be a contravariant functor satisfying the conclusions (i) - (iii) of Lemma \ref{lemma:KoszulSubcomplexExact}.
	Then the complexes 
	\mbox{$\CC^{F^\sigma}_\kbb(m)$} are exact
	for all $m \in M_\bR$ and for all $\sigma\in\normal(\Delta)$.
\end{proposition}

\begin{remark}
	Moreover, this claim remains true if we consider any polyedra compatible with the fan $\normal(\Delta)$, not just lattice polyhedra.
	This generalization will be important for the induction performed in the upcoming proof.
\end{remark}

\begin{proof}
	{\em Step 1.} \label{step1} Assume that $\rankM=\rank M=1$.
	Then $\Delta \subseteq \R$ is a subset of the real line and its normal fan $\normal(\Delta)$ can contain 
	the origin $\sigma = \{0\}$,
	the rays $\sigma = \bR_{\ge 0}$ and $\sigma = \bR_{\le 0}$
	and the real line $\sigma = \bR$ as cones.
	For $\sigma = \bR$ and $\sigma = \{0\}$ the claim is clear from the assumptions.
	We deal with the case $\sigma^\vee = \R_{\ge 0}$ (the case $\sigma^\vee = \R_{\le 0}$ is analogous).
	A polyhedron $F(I')$ compatible with $\normal(\Delta)$ is either empty or an interval of the form $[a, b]$ with $- \infty < a \le b \le \infty$.
	For $F(I') \ne \emptyset$, let $a(I') \in \bR$ denote the start point of the interval $F(I')$ and $b(I') \in \bR \cup \{\infty\}$ the end point. 
	Let \mbox{$A(F) \coloneqq \{a(I') \kst F(I') \ne \emptyset\}$} and $B(F) \coloneqq \{b(I') \kst F(I') \ne \emptyset\}$ be the sets of start and end points, respectively. 
	Since $I$ is finite, $A(F) \subseteq \bR$ and $B(F) \subseteq \bR \cup \{\infty\}$ are finite. \\[1ex]
	We understand the local complex $\CC_\kbb^{F^{\sigma}}$ in terms of global evaluation subcomplexes:
	\begin{align*}
	\CC^{F^\sigma}_\koszulindex(m)
	&=
	\spann \{e_{I'} \kst \#I' = \koszulindex \mbox{ and } m \in F(I') + \bR_{\ge 0} \}  \\
	&=
	\spann \{e_{I'} \kst \#I' = \koszulindex \mbox{ and } a(I')\leq m \} \\
	&=
	\sum_{m'\leq m} \CC^{F}_\koszulindex(m')
	\subseteq \CC^{F}_\koszulindex.
	\end{align*}
	We make three observations:
	\begin{enumerate}
		\item \label{observation1} For $m \in \R$ satisfying $m < b$ for all $b \in B(F)$: 
		\begin{align*}
		\CC^{F^\sigma}_\koszulindex(m) &= \spann \{e_{I'} \kst \#I' = \koszulindex \mbox{ and } a(I')\leq m \} \\
		&= \spann \{e_{I'} \kst \#I' = \koszulindex \mbox{ and } a(I')\leq m \leq b(I')\} = \CC^F_\koszulindex(m).
		\end{align*}
		So for $m \ll 0$ the local complexes $\CC^{F^\sigma}_\kbb(m)$ stabilize to the exact sequence $\CC^F_\kbb(m)$. Recall that $B(F)$ is finite, so $m$ can be chosen small enough.
		\item For $m' \le m$ we have embeddings $\CC^{F^\sigma}_\kbb(m') \hookrightarrow \CC^{F^\sigma}_\kbb(m)$ of complexes.
		\item A non-trivial jump from 
{$\CC^{F^\sigma}_\kbb(<m):=\sum_{m'<m}\CC^{F^\sigma}_\kbb(m')$}
to $\CC^{F^\sigma}_\kbb(m)$ can only occur at a starting point $m = a(I') \in A(F)$ for some $I' \subseteq I$. 
	\end{enumerate}
	We inductively show exactness of $\CC_\kbb^{F^\sigma}(m)$ for all $m \in M_\bR$ by investigating these non-trivial jumps.
	Let $m \coloneqq a(I') \in A(F)$. 
	Because $I$ is finite we can choose an $\epsilon > 0$ such that any $a \in A(F)$ and any $b \in B(F)$ is either equal to $m$ or has distance $|a - m| > \epsilon$ respectively $|b-m| > \epsilon$. 
	Consider the non-trivial embedding
	$
	\CC_\kbb^{F^\sigma}(m - \epsilon) \hookrightarrow \CC_\kbb^{F^\sigma}(m)
	$
	and denote its cokernel by $\CC_\kbb$. 
	This yields an exact sequence
	\begin{equation}\label{equation:inductionExactSequence}
	0 \to \CC_\kbb^{F^\sigma}(m-\epsilon) \rightarrow \CC_\kbb^{F^\sigma}(m) \to \CC _\kbb\to 0.
	\end{equation}
	Because $b(F) \cap [m- \epsilon, m) = \emptyset$, there is also an embedding
	$\CC_\kbb^F(m-\epsilon) \hookrightarrow \CC_\kbb^F(m)$
	for the original functor $F$, also with cokernel $\CC_\kbb$. 
	This gives the short exact sequence
	\begin{equation}
	0 \to \CC_\kbb^F(m-\epsilon) \rightarrow \CC_\kbb^F(m) \to \CC_\kbb \to 0
	\end{equation}
	in which the first two compexes are exact by assumption,
	so $\CC_\kbb$ is also exact.\\
	By \hyperref[observation1]{observation (1)} above we can start from the exact complex $\CC^{F^\sigma}_\kbb(m) = \CC_\kbb^{F}(m)$ for $m < b$ for all $b \in B(F)$ and then inductively use the exact sequence~(\ref{equation:inductionExactSequence})
	with exact complexes $\CC_\kbb^{F^\sigma}(m-\epsilon)$ and $\CC_\kbb$ to obtain that $\CC_\kbb^{F^\sigma}(m)$ is exact for all $m \in M_\R$.
	\\[1ex]
	
	{\em Step 2.} \label{step2}
	Consider a general lattice $M$ of rank $\rankM \in \N_{\ge 1}$ and let $C\subseteq M_\R$ be any ray. 
	Replacing the polyhedra $F(I')$ by the Minkowski sums
	$F(I')+C$ yields the functor $F^C \colon 2^I \to \Pol_{\delta + C}^{+,\bR}$,
	where $\delta = |\normal(\Delta)|\dual$ is the tail cone of $\Delta$.
	Now, for $m\in M_\R$ we restrict the complexes
	$\CC^{F}_\kbb$ and $\CC^{F+C}_\kbb$ to the affine line
	$m+(C-C)$, that is,  consider $F(I') \cap (m+(C-C))$ and $(F(I') + C) \cap (m+(C-C))$.
	By assumption, the complex $\CC^{F\cap (m+(C-C))}_\kbb$ leads to exact evaluation
	subcomplexes 
	$\CC^{F\cap (m+(C-C))}_\kbb(m') = \CC^{F}_\kbb(m')$
	for all \mbox{$m'\in m+(C-C)$}
	{because in this case
		\begin{align*}
			\CC^{F\cap (m+(C-C))}_\koszulindex(m') &= \spann \{e_{I'} \kst \#I' = \koszulindex \mbox{ and } m' \in F(I')\cap(m+(C-C)) \} \\
			&= \spann \{e_{I'} \kst \#I' = \koszulindex \mbox{ and } m' \in F(I') \} = \CC^F_\koszulindex(m').
	\end{align*}}The complex $\CC^{(F+C)\cap (m+(C-C))}_\kbb$ inherits
	this property by \hyperref[step1]{step 1} and we obtain exactness of $\CC_\kbb^{F+C}(m')$ for all $m' \in M_\bR$.
	\\[1ex]
	
	{\em Step 3.} \label{step3}
	If $\sigma^\vee \in \normal(\Delta)$ is an arbitrary (non-trivial) polyhedral cone, then we may apply \hyperref[step2]{step 2} successively to all its
	fundamental rays.
\end{proof}

\subsection{Exactness of the Sequence Associated to a $\Sigma$-Family}
\label{subsection:exactnessExtension}
\begin{theorem}
	\label{theorem:SigmaFamilyGivesExactSequenceOfSheaves}
	Let $\Delta$ be a full-dimensional lattice polyhedron in $M_\bR$ with normal fan $\Sigma \coloneqq \normal(\Delta)$ 
	and $X \coloneqq \PP(\Delta)$ the toric variety given by $\Delta$.
	For \mbox{$S = \{ \nabla_i \kst i \in I\}$} a \mbox{$\Sigma$-family} with $\#I = n$,
	let $F_S \colon 2^I \to \Pol^+(\Sigma)$ be the functor introduced in Definition~\ref{def:functorToSigmaFamily}.
	The complex $\CC^{F_S}_\kbb$ with its exact evaluation subcomplexes $\CC^{F_S}_\kbb(m)$ for $m \in M_\bR$ 
	induces a $T$-equivariant exact sequence of direct sums of nef line bundles on $X$
	\begin{equation}\label{equation:inducedSequencesOfSheavesOnX}
	0 \to \CO_X(\nabla_I) \to \underset{\#I'=n-1}{\oplus} \CO_X(\nabla_{I'})\to \dots \to \underset{i \in I}{\oplus} \CO_X(\nabla_i) \to \CO_X(\nabla) \to 0.
	\end{equation}
\end{theorem}

\begin{proof}
	The evaluation subcomplex $\CC^{F_S}_\kbb(m)$ includes $\C \cdot e_{I'}$ if and only if $m \in \nabla_{I'}$. For $m \in M$, this in turn is equivalent to $\chi^m \in \Gamma(X, \CO_X(D_{\nabla_{I'}}))$. Hence, the evaluation subcomplex $\CC^{F_S}_\kbb(m)$ corresponds to a complex of global sections in degree $m \in M$:
	\begin{equation}
	\CC^{F_S}_\koszulindex(m) \cong \underset{\#I' = \koszulindex}{\oplus} \Gamma(X, \CO_X(\nabla_{I'}))_m.
	\end{equation}
	Summing the $\CC_\kbb^{F_S}(m)$ for lattice points $m \in M$ yields a sequence of global sections:
	\begin{equation}\label{equation:sequenceOfGlobalSections}
	0 \to \Gamma(X, \CO_X(\nabla_I)) \to \dots \to \Gamma(X, \underset{i \in I}{\oplus} \CO_X(\nabla_i)) \to \Gamma(X, \CO_X(\nabla)) \to 0.
	\end{equation}
	The sheaves $\CO_X(\nabla_I)$, $\CO_X(\nabla_i)$, $i \in I$, and $\CO_X(\nabla)$ are globally generated. By \mbox{$T$-equivariance}, they are subsheaves of $j_*\CO_T$ for $j \colon T \hookrightarrow X$ the inclusion of the torus. Hence, the sequence of global sections~(\ref{equation:sequenceOfGlobalSections}) determines a sequence of sheaves:
	\begin{equation}
	0 \to \CO_X(\nabla_I) \to \dots  \to \underset{i \in I}{\oplus} \CO_X(\nabla_i) \to \CO_X(\nabla) \to 0.
	\end{equation}
	For a cone $\sigma \in \Sigma$ and a lattice point $m \in M$, the evaluation subcomplex $\CC^{F^\sigma_S}_\kbb(m)$ includes $e_{I'}$ for $I' \subseteq I$ if and only if $\chi^m \in \Gamma(U_\sigma, \CO_X(\nabla_{I'}))$.
	The sequence of sections over $U_\sigma \subseteq X$  therefore corresponds to the direct sum of the evaluation subcomplexes $\CC^{F_S^\sigma}_\kbb(m)$ for $m \in M$.
	By Proposition~\ref{prop-stayExact} this sequence is exact for each $\sigma \in \Sigma$.
	Hence, the restriction of the sequence of sheaves to each affine chart of the covering $\{U_\sigma\}_{\sigma \in \Sigma}$ is exact and consequently the constructed sequence is exact.
\end{proof}

\section{Displaying $\gExt^1$}
\label{section:displayingExt}
Throughout this section let $X \coloneqq \PP(\Delta)$ be the toric variety over $\bC$ associated to the lattice polyhedron $\Delta \subseteq M_\bR$, whose normal fan $\Sigma \coloneqq \normal(\Delta)$ has convex support of full dimension. 
The (possibly trivial) tail cone of $\Delta$ is \mbox{$\delta \coloneqq \tail(\Delta) = |\Sigma|^\vee$}. 
Let $\Delta^+, \Delta^- \in \Pol^+(\Sigma)$ be lattice polyhedra compatible with $\Sigma$, that is, their normal fans are refined by $\Sigma$. 
These polyhedra correspond to $T$-invariant nef Cartier divisors $D^+ = D_{\Delta^+}$ and $D^- = D_{\Delta^-}$ on $X$. 
We study the space of extensions
\begin{equation}
\gExt(\Delta^-, \Delta^+) \cong \gExt^1(\Delta^-, \Delta^+) \coloneqq \gExt^1(\CO_X(\Delta^-), \CO_X(\Delta^+)),
\end{equation}
that is, extensions of the line bundle $\CO_X(\Delta^-)$ by the line bundle $\CO_X(\Delta^+)$.
More specifically, we study $T$-equivariant extension sequences. These are elements in $\gExt(\Delta^-, \Delta^+)_0$. To understand this space we start with specific extension sequences induced by inclusion/exclusion sequences of polyhedra, such as the sequence considered in the Introduction~(\ref{subsection:exampleF1}). For an $n$-dimensional $\gExt\big(\Delta^-, \Delta^+\big)_0$ all extensions are encoded in a single sequence of the form 
$
0 \to \CO_X(\Delta^+)^n \to \CH \to \CO_X(\Delta^-) \to 0.
$
We will show that this is the universal extension sequence for  $\gExt\big(\Delta^-, \Delta^+\big)_0$.
First, we show that it can be constructed from the exact sequence of sheaves associated to a $\Sigma$-family (Theorem~\ref{theorem:SigmaFamilyGivesExactSequenceOfSheaves}). 
We then trace the sequence along the identifications:
\begin{align*} \label{identifications}
\gExt\big(\Delta^-, (\Delta^+)^n\big) 
& \overset{(1)}{\cong} \gExt\big(\CO_X, \CO_X(\Delta^+ - \Delta^-)^n\big) \\ 
& \overset{(2)}{=} \gH^1\big(X, \CO_X(\Delta^+ - \Delta^-)\big)^n \\
& \overset{(3)}{\cong} \big(\oplus_{m \in M} \tH^0(\Delta^- \setminus (\Delta^+ - m))\big)^n.
\end{align*}
Note that all of the above groups are $M$-graded and all identifications respect these $M$-gradings. 
Since we consider equivariant extensions, we will obtain an $n$-tuple of elements in $\gH^1(X, \CO_X(\Delta^+ - \Delta^-))_0 \cong \tH^0(\Delta^- \setminus \Delta^+)$, without an integral shift.

\subsection{Inclusion of Polyhedra}
\label{subsection:inclusionOfPolyhedra}
Since we are relating $\gExt\big(\CO_X(\Delta^-), \CO_X(\Delta^+)\big)_0$ to the reduced singular cohomology group $\tH^0(\Delta^- \setminus \Delta^+)$, the easiest case is to assume $\Delta^+$ to be contained in $\Delta^-$.
The inclusion/exclusion sequence of polyhedra is obtained by covering $\Delta^-$ by certain polyhedra $\nabla_i$ that intersect (pairwise) in $\Delta^+$. These $\nabla_i$ are obtained by taking the unions of the connected components $C_0, \dots, C_n$ of the set-theoretic difference $\Delta^-\setminus \Delta^+$ with $\Delta^+$, that is,
$\nabla_i \coloneqq C_i \cup \Delta^+, \, i = 0, \dots, n.$

\begin{proposition}
	\label{proposition:nablaAreCompatiblePolyhedra}
	For two lattice polyhedra $\Delta^+ \subseteq \Delta^- \in \Pol^+(\Sigma)$ and each connected component $C = C_i$ of $\Delta^- \setminus \Delta^+$, the union $\nabla = C \cup \Delta^+$ is again a lattice polyhedron.
\end{proposition}

\begin{proof}
	{It is easy to see that $\nabla$ is closed.}
	We first show that $\nabla$ is convex.
	Assume that $x,y\in \nabla$.
	We know that $\ko{xy}\subseteq\Delta^-$, and this line segment might touch
	$\Delta^+$ or not. 
	If not, then $\ko{xy}\subseteq(\Delta^-\setminus\Delta^+)$.
	Since $\ko{xy}$ is connected, it then has to be contained in a single
	connected component of $\Delta^-\setminus\Delta^+$.
	Since $x\in \nabla = C \cup \Delta^+$, we know that $x\in C$.
	Hence, $\ko{xy}\subseteq C$.
	If $\ko{xy}\cap\Delta^+\neq \emptyset$, then this set is a closed subsegment
	$\ko{x'y'}\subseteq\ko{xy}$. In particular, the half open
	ends $\ko{xx'}$ and $\ko{y'y}$ (excluding $x'$ and $y'$) belong to
	$\Delta^-\setminus\Delta^+$, and the same argument as in the first case
	applies again: Since $x,y \in \nabla = C \cup \Delta^+$, both of these half open ends $\ko{xx'}$ and $\ko{y'y}$ lie in $C$. The segment $\ko{xy}$ thus lies in $\nabla = C \cup \Delta^+$.
	\\[1ex]
	Next we realize that vertices of $\nabla$ are vertices of $\Delta^+$ or of $\Delta^-$.
	We start with the case of $\Delta^+$ and $\Delta^-$ being compact. 
	Choose a generic regular triangulation induced from some map $\omega \colon \vertset(\Delta^+) \to \bR$ and then extend it generically to $\vertset(\Delta^-) \setminus \vertset(\Delta^+)$ with sufficiently independent heights. This yields a triangulation $\Delta^- = \bigcup_{i \in I} \Delta^i$ 
	that restricts to a triangulation $\Delta^+ = \bigcup_{j \in J \subseteq I} \Delta^j$ 
	that uses only the vertices of $\Delta^-$ and $\Delta^+$. In particular all $\Delta^i$ are lattice simplices.
	The union $\nabla = C \cup \Delta^+$ is a union of lattice simplices in $I$. In particular all of its vertices are lattice points.\\[1ex]
	If $\Delta^+$ and $\Delta^-$ are not compact but have the same tail cone $\delta$ we choose a half-space $H$ such that all vertices of $\Delta^+$ and $\Delta^-$ are contained in the interior of $H$. Then we can write $\Delta^+ = P^+ + \delta$ and $\Delta^- = P^- + \delta$ with $P^+ \coloneqq \Delta^+ \cap H$, $P^- \coloneqq \Delta^- \cap H$. The previous discussion applied to the polytopes $P^+$ and $P^-$, all of whose relevant vertices are integral (not necessarily those lying on the boundary of $H$, but they do not yield vertices of $C \cup \Delta^+$), yields that all vertices of $\nabla = C \cup \Delta^+$ are integral. 
\end{proof}

\begin{corollary}
	\label{corollary:componentsFormSigmaFamily}
	For $\Delta^+ \subseteq \Delta^- \in \Pol^+(\Sigma)$, the set
	$
	S \coloneqq \{\nabla_i \kst i \in \{0, \dots, n\} \eqqcolon I\}
	$
	with $\nabla_i \coloneqq C_i \cup \Delta^+$ for the connected components $C_0, \dots, C_n$ of $\Delta^- \setminus \Delta^+$ yields a $\Sigma'$-family in the sense of Definition \ref{def:SigmaFamily} for some refinement $\Sigma' \le \Sigma$.
\end{corollary}

\begin{proof}
	By the previous Proposition \ref{proposition:nablaAreCompatiblePolyhedra} the $\nabla_i$ are lattice polyhedra. 
	The tail cone of each $\nabla_i$ is $|\Sigma|\dual$. 
	Hence, each normal fan $\normal(\nabla_i)$ has the same support as $\Sigma$. 
	Since we are only dealing with finitely many $\nabla_i$ there is a common refinement $\Sigma'$ of $\Sigma$ and all $\normal(\nabla_i)$. 
	Then $\nabla_i \in \Pol^+(\Sigma')$. 
	The conditions for a $\Sigma'$-family follow from the assumptions on $\Sigma'$ and $\nabla_{I'} = \Delta^+ \in \Pol^+(\Sigma')$ for any $I' \subseteq I$ with $\#I' \ge 2$.
\end{proof}

\subsubsection{Refining the Fan}
\label{subsubsection:refiningTheFan}
In many cases the lattice polyhedra $\nabla_i$ built from the components of $\Delta^- \setminus \Delta^+$ will already be compatible with the fan $\Sigma$ we started with. We believe that if $\Delta^+$ and $\Delta^-$ are sufficiently ample, this will always be the case. However, in general, if the $\nabla_i$ are not compatible with $\Sigma$, we can refine $\Sigma$ to a fan $\Sigma'$ with the same support such that all $\nabla_i$ are compatible with $\Sigma'$, as in Corollary~\ref{corollary:componentsFormSigmaFamily}. 
This induces a proper birational toric morphism 
$
\pi \colon X' = \toric(\Sigma') \to \toric(\Sigma) = X.
$
{%
This morphism 
satisfies that 
$\pi_* \CO_{X'} = \CO_X$ and $R^\nu \pi_*\CO_{X'}= 0$ for all $\nu > 0$.%
}
The projection formula implies that $\pi_*\pi^*\CF = \CF$ for a locally free sheaf $\CF$ on $X$.
For a short exact sequence of sheaves on $X'$ of the form
\begin{equation}
0 \to \pi^*\CF_1 \to \CE \to \pi^*\CF_2 \to 0,
\end{equation}
derived pushforward $R\pi_*$ yields a long exact sequence of sheaves on $X$:
\begin{equation}
0 \to \underbrace{\pi_*\pi^*\CF_1}_{= \CF_1} \to \pi_*\CE \to \underbrace{\pi_*\pi^*\CF_2}_{= \CF_2} \to \underbrace{R^1\pi_*\pi^*\CF_1}_{=\CF_1 \otimes R^1\pi_*\CO_{X'} = 0} \to \dots,
\end{equation}
which turns out to be a short exact sequence by the derived version of 
the projection formula.
{%
In particular, the vanishing $R^\nu\pi_*\pi^*\CF_j=0$ for $\nu\geq 1$
implies $R^\nu\pi_*\CE=0$, too.
Hence, using both vanishings,}
we have for all $i \ge 0$ and for $j = 1, 2$ the isomorphisms
\begin{equation}
\gH^i(X', \CE) \cong \gH^i(X, \pi_*\CE) \text{ and } \gH^i(X', \pi^*\CF_j) \cong \gH^i(X, \pi_*\pi^*\CF_j) = \gH^i(X, \CF_j).
\end{equation}
In the study of 
$
\gExt(\Delta^-, \Delta^+) \cong \gH^1(X, \CO_X(\Delta^+ - \Delta^-)) \cong \gH^1(X', \CO_{X'}(\Delta^+ - \Delta^-))
$
we can therefore pull back and push foward along $\pi$ without impacting the extension or cohomology classes and groups. 
In the following, we will assume for simplicity that the fan $\Sigma$ is already refined enough so that all  $\nabla_i$ are compatible with it.

\subsubsection{Two Components}
\label{subsubsection:twoComponents}
We dive into Theorem \ref{theorem:SigmaFamilyGivesExactSequenceOfSheaves} in the case where $\Delta^- \setminus \Delta^+$ consists of two components $C_0$ and $C_1$ and we assume $\nabla_0 = C_0 \cup \Delta^+$ and $\nabla_1 = C_1 \cup \Delta^+$ to be compatible with the fan $\Sigma$. The complex $\CC_\kbb^{F_S}$ associated to \mbox{$S = \{\nabla_0, \nabla_1\}$} is
\begin{equation}
0 \to \bC \cdot e_I \xrightarrow{\begin{pmatrix} -1 \\ 1 \end{pmatrix}} \bC \cdot e_0 \oplus \bC \cdot e_1 \xrightarrow{\begin{pmatrix} 1 & 1 \end{pmatrix}} \bC \cdot e_\emptyset \to 0.
\end{equation}
By Theorem \ref{theorem:SigmaFamilyGivesExactSequenceOfSheaves} the $C_\kbb^{F_S}(m)$ for $m \in M$ induce a $T$-equivariant exact sequence 
\begin{equation}\label{equation:extensionSequenceTwoComponents}
0 \to \CO_X(\Delta^+) \xrightarrow{} \CO_X(\nabla_0) \oplus \CO_X(\nabla_1) \xrightarrow{} \CO_X(\Delta^-) \to 0.
\end{equation}
\begin{theorem}
	\label{theorem:ExtComp}
	Via the identification
	$
	\gExt \big(\CO_X(\Delta^-),\CO_X(\Delta^+)\big)_0=
	\tH^{0}\big(\Delta^- \setminus\Delta^+\big),
	$
	the exact extension sequence~(\ref{equation:extensionSequenceTwoComponents})
	corresponds to the reduced singular $0$-th cohomology class $[C_1] \in \tH^0\big(\Delta^- \setminus \Delta^+\big)$, which is equal to the class $-[C_0] \in \tH^0\big(\Delta^- \setminus \Delta^+\big)$.
\end{theorem}
\begin{proof}
	We follow the steps mentioned at the start of section \ref{section:displayingExt}. We \hyperref[identifications]{first} tensor the sequence~(\ref{equation:extensionSequenceTwoComponents})
	with $\CO_X(\Delta^-)^{-1}$ and set $\CL \coloneqq \CO_X(\Delta^+ - \Delta^-)$ and $\CE_i \coloneqq \CO_X(\nabla_i - \Delta^-)$ for $i = 0, 1$ to obtain an extension sequence
	\begin{equation}\label{equation:tensoredSequenceTwoComponents}
	0 \to \CL \xrightarrow{} \CE_0 \oplus \CE_1 \xrightarrow{} \CO_X \to 0
	\end{equation}
	in $\gExt(\CO_X, \CL)_0$.
	Consider the associated long exact sequence in cohomology:
	\begin{equation}\label{equation:longExactSequenceSheafCohomology}
	0 \to \Gamma(X, \CL) \to \Gamma(X, \CE_0) \oplus \Gamma(X, \CE_1) \to \Gamma(X, \CO_X) \xrightarrow[1 \mapsto \eta]{d} \gH^1(X, \CL) \to \dots
	\end{equation}
	The image of the extension sequence~(\ref{equation:extensionSequenceTwoComponents}) in $\gH^1(X, \CL)$ under \hyperref[identifications]{identification (2)} is the image $\eta$ of $1 \in \Gamma(X, \CO_X)$ under $d$.
	Note that $\eta \in \gH^1(X, \CL)_0$ is of degree $0$. \\[1ex]
	To understand \hyperref[identifications]{identification (3)} we translate $\eta \in \gH^1(X, \CL)$ to an element in the \v{C}ech cohomology group $\check{\gH}^1(\CU, \CL)$ for $\CU = \{U_{\sigma_i}\}_{\sigma_i \in \Sigma_{\max}}$ with $U_{\sigma_i} = \Spec \bC [\sigma_i\dual \cap M]$ the standard toric affine open covering of $X$.
	The long exact sequence in cohomology~(\ref{equation:longExactSequenceSheafCohomology}) corresponds to a long exact sequence of \v{C}ech cohomology groups
	\begin{equation}\label{equation:longExactSequenceCechCohomology}
	0 \to \check{\gH}^0(\CU, \CL) \to \check{\gH}^0(\CU, \CE_0) \oplus \check{\gH}^0(\CU, \CE_1) \to \check{\gH}^0(\CU, \CO_X) \xrightarrow[1 \mapsto \eta]{d} \check{\gH}^1(\CU, \CL) \to \dots
	\end{equation}
	 Order the maximal cones $\sigma \in \Sigma_{\max}$ so that $1 \in \CE_0(U_i)$ for $i \in \{0, \dots, l\}$ and $1 \in \CE_1(U_i)$ for $i \in \{l+1, \dots, m\}$. {This is possible because the map of sheaves 
$\CE_0 \oplus \CE_1 \xrightarrow{} \CO_X$ is surjective, each 
$U_i:=U_{\sigma_i}$ is affine,
and everything is $T$-equivariant. Note that this is not a dichotomy;
it might happen that
$1 \in \CE_0(U_i)\cap\CE_1(U_i)$ for some $i$. Using this order,}
	the boundary homomorphism $d$ the element \mbox{$(1, \dots, 1) \in \prod_i \CO_X(U_i)$} maps to the class $\eta \in \check{\gH}^1(\CU, \CL)$ of 
	\begin{equation}\label{equation:elementOfChechComplex}
	(\underbrace{0, \dots, 0}_{i < j \le l}, \underbrace{1, \dots, 1}_{i \le l < j}, \underbrace{0, \dots, 0}_{l < i < j}) \in \prod_{i < j} \CL(U_{ij}).
	\end{equation}
	Understanding $\eta$ in terms of \v{C}ech cohomology is convenient because by \cite{dop} the degree $0$ part of the \v{C}ech complex $\CC^\kbb(\CU, \CL)$ giving $\check{\gH}^\kbb(\CU, \CL)$
	is the same as a \v{C}ech complex $\CC^\kbb(\Delta^-, \CS)$ defined in \cite[(3.4)]{dop} giving the relative singular cohomology $\gH^\kbb(\Delta^-, \Delta^- \setminus \Delta^+)$. 
	It is defined as 
	$\CC^p(\Delta^-, \CS) \coloneqq \prod_{i_0 < \dots < i_p} \gH^0(\Delta^-, S(\sigma_{i_0} \cap \dots \cap \sigma_{i_p}))$ 
	for
	$S(\sigma) \coloneqq \Delta^- \setminus (\Delta^+ + \sigma^\vee).$
	Investigating the groups $\gH^0(U_\sigma, \CL)_0 = \gH^0(\Delta^-, S(\sigma))$ and $\gH^0(U_\sigma, \CE_i)_0$, $i = 0,1$, as in \cite[(3.2)]{dop}, one realizes that for $\sigma \in \Sigma_{\max}$ with $S(\sigma) \ne \emptyset$, $1 \in \CE_0(U_\sigma)$ implies $S(\sigma) \subseteq C_0$ and $1 \in \CE_1(U_\sigma)$ implies $S(\sigma) \subseteq C_1$. Here $C_0$ and $C_1$ denote the two connected components of $\Delta^- \setminus \Delta^+$.\\[1ex]
	The element $\eta \in \check{\gH}^1(\CU, \CL)_0 \cong \gH^1(\Delta^-, \Delta^-\setminus\Delta^+)$ 
	can be lifted to 
	$\gH^0(\Delta^- \setminus \Delta^+)$ by considering the long exact sequence of the pair \mbox{$(\Delta^-, \Delta^-\setminus\Delta^+)$}. 
	In terms of \v{C}ech complexes the boundary morphism \mbox{$\gH^0(\Delta^- \setminus \Delta^+) \xrightarrow{d} \gH^1(\Delta^-, \Delta^-\setminus\Delta^+)$} is given by the snake lemma.
	One shows that $\eta$ can be lifted to the cocycle which is $1$ on the component $C_1$ of \mbox{$\Delta^- \setminus \Delta^+$} and $0$ on the component $C_0$. It can also be lifted to the cocycle which is $0$ on the component $C_1$ and $-1$ on the component $C_0$. Modulo $\gH^0(\Delta^-)$ these cocycles are equivalent and we have found the image of sequence~(\ref{equation:extensionSequenceTwoComponents}) to be \mbox{$[C_1] = -[C_0] \in \tH^0(\Delta^- \setminus \Delta^+)$.}
\end{proof}

\subsubsection{More than Two Components}
\label{subsubsection:manyComp}
We now deal with the case where $\Delta^-\setminus\Delta^+$ consists of $n+1$ connected components $C_0,\ldots,C_n$ for some $n \ge 2$. Set $\nabla_i \coloneqq C_i \cup \Delta^+$ for $i = 0, \dots, n$ and note that $\bigcup_{i = 0}^n \nabla_i = \Delta^-$ and $\nabla_{I'} = \bigcap_{i \in I'} \nabla_i = \Delta^+$ for $I'\subseteq I$ of cardinality $k$ with $2 \le k \le n+1$. 
By subsection~(\ref{subsubsection:refiningTheFan}) we may assume that the set $S \coloneqq \{\nabla_i \kst i \in \{0, \dots, n\}\}$ forms a $\Sigma$-family. 
The exact sequence of sheaves induced from this $\Sigma$-family as in Theorem \ref{theorem:SigmaFamilyGivesExactSequenceOfSheaves} looks as follows:
\begin{equation}\label{equation:extensionSequenceMoreComponents}
0 \to \CO_X(\Delta^+) \to \dots \to \CO_X(\Delta^+)^{\binom{n+1}{2}} \to
\oplus_{i=0}^n \,\CO_X(\nabla_i) \to \CO_X(\Delta^-)\to 0.
\end{equation}
We will replace this sequence~(\ref{equation:extensionSequenceMoreComponents}) by a quasi-isomorphic short exact sequence. To this end, consider the exact Koszul complex $\bigwedge^\kbb\bC^{n+1} \otimes \CO_X(\Delta^+)$:
\begin{equation}\label{equation:KoszulComplex}
\color{darkgreen} 0 \to \CO_X(\Delta^+) \xrightarrow{\bar{d}_{n+1}} \dots \xrightarrow{\bar{d}_{3}} \CO_X(\Delta^+)^{\binom{n+1}{2}} \color{black} \xrightarrow{\bar{d}_{2}} \CO_X(\Delta^+)^{n+1} \xrightarrow{\bar{d}_{1}} \CO_X(\Delta^+) \to 0,
\end{equation}
where $\bar{d}_i = d_i \otimes \id_{\CO_X(\Delta^+)}$ for $d_i$ the $i$-th differential in $\bigwedge^\kbb\bC^{n+1}$.
The green part is identical to part of sequence~(\ref{equation:extensionSequenceMoreComponents}). 
We obtain a quasi-isomorphism 
{%
from sequence~(\ref{equation:extensionSequenceMoreComponents}):
$$
\xymatrix@R=4.5ex@C=1.7em{
0 \ar[r] & \CO_X(\Delta^+) \ar[r]       &  \dots \ar[r]  &
\CO_X(\Delta^+)^{\binom{n+1}{2}} \ar[r] & 
\oplus_{i=0}^n \,\CO_X(\nabla_i) \ar[r] & \CO_X(\Delta^-) \ar[r] &  0
\\
0 \ar[r] & 0 \ar[r] \ar@{<-}[u] & \dots \ar[r]  & \ker(\bar{d}_1) \ar[r] 
\ar@{<<-}[u] & \oplus_{i=0}^n \,\CO_X(\nabla_i) 
\ar[r] \ar@{=}[u] & \CO_X(\Delta^-)\ar@{=}[u] \ar[r] &  0. 
}
$$
}%
Denoting $K \coloneqq \ker(\bar{d}_1) = \ker(d_1) \otimes \CO_X(\Delta^+)$, we obtain a short exact sequence 
\begin{equation}\label{equation:inducedSESMoreComponentsWithKernel}
0 \to K \to \oplus_{i=0}^n \,\CO_X(\nabla_i) \to \CO_X(\Delta^-)\to 0.
\end{equation}
We now choose the set $\{e_i - e_0 \kst i \in \{1, \dots, n\}\}$ with respect to the standard basis $\{e_0, \dots, e_n\}$ for $\bC^{n+1}$ as a basis for $\ker(d_1)$. This induces an isomorphism $K \cong \CO_X(\Delta^+)^n$ under which sequence~(\ref{equation:inducedSESMoreComponentsWithKernel}) corresponds to the short exact sequence
\begin{equation}\label{equation:inducedSESMoreComponents}
0 \to \CO_X(\Delta^+)^{n} \to
\oplus_{i=0}^n \,\CO_X(\nabla_i) \to \CO_X(\Delta^-)\to 0,
\end{equation}
where the maps can be thought of as 
\begin{gather}
	\CO_X(\Delta^+)^{n}  \xrightarrow{\begin{pmatrix*} -1 & -1 & \dots & -1 \\ 1 & 0 & \dots & 0 \\ 0 & 1 & \dots & 0 \\ \vdots & \ddots  & \ddots & \vdots \\ 0 & \dots & 0 & 1 \end{pmatrix*}}  \oplus_{i=0}^n \,\CO_X(\nabla_i), \\
	\oplus_{i=0}^n \,\CO_X(\nabla_i)  \xrightarrow{\begin{pmatrix*} 1 & 1 & \hspace{0.6cm} \dots \hspace{0.6cm} & 1\end{pmatrix*}}  \CO_X(\Delta^-).
\end{gather}

\begin{theorem}
	\label{theorem:ExtMultipleComponents}
	Via the identification
	$
	\gExt\big(\CO_X(\Delta^-),\CO_X(\Delta^+)^n\big)_0=
	\tH^{0}\big(\Delta^- \setminus\Delta^+\big)^n,
	$
	the extension sequence~(\ref{equation:inducedSESMoreComponents})
	corresponds to the $n$-tuple
	$([C_1],\ldots,[C_n]) \in \tH^{0}\big(\Delta^- \setminus\Delta^+\big)^n$ of reduced singular $0$-th cohomology classes.
\end{theorem}

\begin{proof}
	The steps in this case are very similar to the case of two components. 
	{The extension sequence
	$$
	0 \to \CL^n \rightarrow
	 \oplus_{i = 0}^{n} \CE_i \rightarrow 
	 \CO_X \to 0
	$$
	corresponds to the image $\eta$ of $1 \in \Gamma(X, \CO_X)$ in $\gH^1(X, \CL^n) = \gH^1(X, \CL)^n$ under the long exact sequence in cohomology associated to the short exact extension sequence. In our case $\CL = \CO_X(\Delta^+ - \Delta^-)$ and $\CE_ i = \CO_X(\nabla_i - \Delta^-)$.
	Again, this image is understood using the sequence of \v{C}ech complexes, only looking at degree 0:
$$
\xymatrix@R=4.5ex@C=1.7em{
	& 0 \ar[d] & 0 \ar[d] & 0 \ar[d] & \\
	0 \ar[r] &
	\prod_{i}\CL(U_i)^n \ar[r] \ar@{->}[d] & 
	\prod_{i}(\oplus_{s = 0}^n \CE_s(U_i)) \ar[r] \ar@{->}[d] & 
	\prod_{i}\CO_X(U_i) \ar[r] \ar@{->}[d] & 0\\
	0 \ar[r] &	
	\prod_{i < j} \CL(U_{ij})^n \ar[r] \ar@{->}[d] &
	\prod_{i < j}(\oplus_{s=0}^n \CE_s(U_{ij})) \ar[r] \ar@{->}[d] &
	\prod_{i < j} \CO_X(U_{ij}) \ar[r] \ar@{->}[d] & 0\\
	0 \ar[r] &		
	\prod_{i < j < k} \CL(U_{ijk})^n \ar[r] \ar[d] &
	\prod_{i < j < k}(\oplus_{s=0}^n \CE_s(U_{ijk})) \ar[r] \ar[d] &
	\prod_{i < j < k} \CO_X(U_{ijk}) \ar[r] \ar[d] & 0\\
	& \vdots & \vdots & \vdots &
}
$$
	We lift the element $(1, \dots, 1) \in \prod_i \CO_X(U_i)$ to $\prod_{i}(\oplus_{s = 0}^n \CE_s(U_i))$, then map it to $\prod_{i < j}(\oplus_{s = 0}^n \CE_s(U_ij))$ via the \v{C}ech differential and lift it to $\prod_{i < j} \CL(U_{ij})^n$.
	The resulting element $(\eta_{ij})_{i<j} \in \prod_{i < j} \CL(U_{ij})^n$ can be viewed as an element of $\prod_{i < j}\gH^0(\Delta^-,S(\sigma_i \cap \sigma_j))^n$. The preimage of $[\eta]$ under the map $d^n$ in the long exact sequence
	$$
	0 \to \gH^0(\Delta^-, \Delta^- \setminus \Delta^+)^n \to \gH^0(\Delta^-)^n \to \gH^0(\Delta^- \setminus \Delta^+)^n \xrightarrow{d^n} \gH^1(\Delta^-, \Delta^- \setminus \Delta^+)^n \to 0.
	$$ is the desired element $([C_1],\ldots,[C_n]) \in \tH^{0}\big(\Delta^- \setminus\Delta^+\big)^n$.}
\end{proof}
Recall that $\gExt^1(\CO_{X}(\Delta^-), -)$ is a covariant functor, where a map of extensions is induced by a pushout.
\begin{corollary}
	\label{cor:singleExtensions}
	Given the extension sequence~(\ref{equation:inducedSESMoreComponents}) 
	in $\gExt\big(\CO_X(\Delta^-),\CO_X(\Delta^+)^n\big)_0$, we obtain $n$ extensions in $\gExt\big(\CO_X(\Delta^-),\CO_X(\Delta^+)\big)_0$,
	one for each $i = 1, \dots, n$:
	$$
	\xymatrix@R=4.5ex@C=1.7em{
		0 \ar[r] & 
		\CO_X(\Delta^+)^n \ar[r] \ar@{->>}[d]^{pr_i} & 
		\oplus_{i = 0}^n \CO(\nabla_i) \ar[r] \ar@{->}[d] &
		\CO_X(\Delta^-) \ar[r] \ar@{=}[d]& 0
		\\
		0 \ar[r] & 
		\CO_X(\Delta^+) \ar[r] & \CH_i \ar[r] &
		\CO_X(\Delta^-) \ar[r] & 0,
	}
	$$
	where $\CH_i$ is the pushout of the left square.\\
	Via the identification
	$
	\gExt\big(\CO_X(\Delta^-), \CO_X(\Delta^+)\big)_0 =
	\tH^{0}\big(\Delta^- \setminus\Delta^+\big),
	$
	the $i$-th extension 
	\begin{equation}
	0 \to
	\CO_X(\Delta^+) \to\CH_i \to
	\CO_X(\Delta^-) \to 0
	\end{equation}
	corresponds to the class $[C_i] \in \tH^0\big(\Delta^- \setminus\Delta^+\big)$.
	In particular, the $n$ extensions for $i \in \{1, \dots, n\}$ form a basis of $\gExt\big(\CO_X(\Delta^-),\CO_X(\Delta^+)\big)_0$.
\end{corollary}
\begin{proof}
	This follows from functoriality of $\gExt(\CO_X(\Delta^-), -)$ and naturality of the isomorphisms we identify along. The $i$-th projection $\CO_X(\Delta^+)^n \to \CO_X(\Delta^+)$ induces the $i$-th projection $\gExt\big(\CO_X(\Delta^-), \CO_X(\Delta^+)\big)_0^n \to \gExt\big(\CO_X(\Delta^-),\CO_X(\Delta^+)\big)_0$. By the previous Theorem \ref{theorem:ExtMultipleComponents} this yields $[C_i] \in \tH^0(\Delta^- \setminus \Delta^+)$.
\end{proof}
\begin{remark}
	To obtain $[C_0]$ one has to replace $\pr_i$ by $(-1, \dots, -1)$.
\end{remark}

\subsubsection{The Universal Extension}
\label{subsubsection:universalExtension}
For two $\CO_X$-modules $\CF$ and $\CG$ and their (finite-dimensional) space of extensions $E \coloneqq \gExt(\CF, \CG)$, a \emph{universal extension} is a short exact sequence of the form
$
0 \to \CG \to \CH \to \CF \otimes E \to 0
$ 
such that for any $t \in E$ the induced pullback sequence
is $t \in E = \gExt(\CF,\CG)$. Equivalently, it is a short exact sequence 
$
0 \to \CG \otimes E\dual \to \CH' \to \CF \to 0,
$
where $E\dual = \Hom_\bC(E, \bC)$, such that for any $t \in E$ the sequence
induced by the pushout along $t \in E \cong \Hom_\bC(E\dual, \bC)$ is $t \in E = \gExt(\CF,\CG)$.
Analogously one defines a \emph{universal $T$-equivariant extension} $0 \to \CG \to \CH \to \CF \otimes E_0 \to 0$ or $0 \to \CG \otimes E_0\dual \to \CH' \to \CF \to 0$ for $E_0 \coloneqq \gExt(\CF, \CG)_0$.

\begin{theorem}
	The extension sequence~(\ref{equation:inducedSESMoreComponentsWithKernel}) in $\gExt\big(\CO_X(\Delta^-), K \big)_0$ is a universal extension for $\gExt\big(\CO_X(\Delta^-), \CO_X(\Delta^+)\big)_0$.
\end{theorem}

\begin{proof}
	The $\bC$-vector space $\gExt(\CO_X(\Delta^-), \CO_X(\Delta^+))_0 \cong \tH^0(\Delta^- \setminus \Delta^+)$ is the cokernel of the homomorphism 
	$
	\gH^0(\{\cdot\}) \to \gH^0(\Delta^- \setminus \Delta^+).
	$
	The components $C_0, \dots, C_n$ provide a natural basis for $\gH^0(\Delta^- \setminus \Delta^+)$ with dual basis $C_0\dual, \dots, C_n\dual$ of $\gH^0(\Delta^- \setminus \Delta^+)\dual$. 
	The quotient $\tH^0(\Delta^- \setminus \Delta^+) = \frac{\gH^0(\Delta^- \setminus \Delta^+)}{\gH^0(\{\cdot\})}$ is generated by $[C_0], \dots, [C_n]$ and subject to the relation $[C_0] + \cdots +  [C_n] = 0$. 
	We choose $\{C_1\dual - C_0\dual, \dots, C_n\dual - C_0\dual\}$ as a basis for the dual vector space $\tH^0(\Delta^- \setminus \Delta^+)\dual = \ker(\gH^0(\Delta^- \setminus \Delta^+)\dual \to \gH^0(\{\cdot\})\dual)$.
	This basis is dual to the basis $\{[C_1], \dots, [C_n]\}$ of $\tH^0(\Delta^- \setminus \Delta^+)$. 
	The isomorphism $\tH^0(\Delta^- \setminus \Delta^+)\dual \xrightarrow{\cong} \ker(d_1)$ maps the $i$-th basis element $C_i\dual - C_0\dual$ to the $i$-th basis element $e_i - e_0$. Recall that $K = \ker(d_1) \otimes \CO_X(\Delta^+)$. \\[1ex]
	We now check the universal property on the basis $\{[C_1], \dots, [C_n]\}$ of $\tH^0(\Delta^- \setminus \Delta^+)$. By definition of the dual and double dual, $[C_i] \in \tH^0(\Delta^- \setminus \Delta^+) \cong \Hom_\bC(\tH^0(\Delta^- \setminus \Delta^+)\dual, \bC)$ corresponds to the projection to the $i$-th coordinate \mbox{$\pr_i \colon \tH^0(\Delta^- \setminus \Delta^+)\dual \cong \bC^n \to \bC$}.
	By Corollary~\ref{cor:singleExtensions} the pushout along the $i$-th projection homomorphism gives an extension in $\gExt\big(\CO_X(\Delta^-), \CO_X(\Delta^+)\big)_0 \cong \tH^0(\Delta^- \setminus \Delta^+)$ which corresponds precisely to the basis element \mbox{$[C_i] \in \tH^0(\Delta^- \setminus \Delta^+)$.}
\end{proof}

\subsubsection{The Cremona Example}
\label{subsubsection:cremonaExample}
Let $X$ be the graph of the Cremona transformation $\PP^2\rato\PP^2$.
In toric language, $X= \PP(H) = \toric(\Sigma)$, where $\Sigma = \normal(H)$:
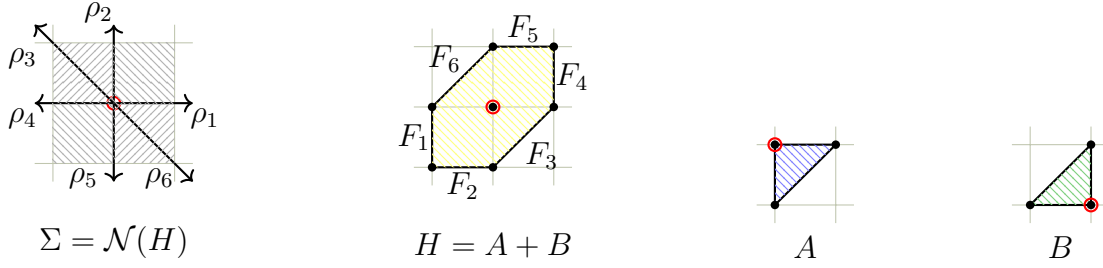
$$
\newcommand{\scaleA}{0.8}
\newcommand{\spaceA}{\hspace*{5em}}
\begin{tikzpicture}[scale=\scaleA]
\draw[color=oliwkowy!40] (-1.3,-1.3) grid (1.3,1.3);
\draw[thick, color=red]
(0,0) circle (3pt);
\draw[thick,  color=black]
(0.0,-2.3) node{$\Sigma=\normal(H)$};
\draw[thick,  color=black, ->]
(0,0) -- (-1.3,0);
\draw[thick,  color=black, ->]
(0,0) -- (1.3,0);
\draw[thick,  color=black, ->]
(0,0) -- (0,-1.3);
\draw[thick,  color=black, ->]
(0,0) -- (0,1.3);
\draw[thick,  color=black, ->]
(0,0) -- (-1.3,1.3);
\draw[thick,  color=black, ->]
(0,0) -- (1.3,-1.3);
\fill[pattern color=black!30, pattern=north west lines]
(0,0) -- (1,0) -- (1,1) -- (0,1);
\fill[pattern color=black!30, pattern=north east lines]
(0,0) -- (0,1) -- (-1,1) -- (0, 0);
\fill[pattern color=black!30, pattern=north east lines]
(0,0) -- (-1,1) -- (-1,0) -- (0, 0);
\fill[pattern color=black!30, pattern=north west lines]
(0,0) -- (0,-1) -- (-1,-1) -- (-1,0);
\fill[pattern color=black!30, pattern=north east lines]
(0,0) -- (0,-1) -- (1,-1) -- (0, 0);
\fill[pattern color=black!30, pattern=north east lines]
(0,0) -- (1,0) -- (1,-1) -- (0, 0);
\draw[thick,  color=black]
(1.5,-0.25) node{$\rho_1$} (-0.25,1.5) node{$\rho_2$} 
(-1.5,0.75) node{$\rho_3$} (-1.5,-0.25) node{$\rho_4$}
(-0.5,-1.25) node{$\rho_5$} (0.75,-1.25) node{$\rho_6$};
\end{tikzpicture}
\spaceA
\begin{tikzpicture}[scale=\scaleA]
\draw[color=oliwkowy!40] (-1.3,-1.3) grid (1.3,1.3);
\draw[thick, color=black]
(-1,-1) -- (0,-1) -- (1,0) -- (1,1) -- (0,1) -- (-1,0) -- (-1,-1);
\fill[pattern color=yellow!50, pattern=north west lines]
(-1,-1) -- (0,-1) -- (1,0) -- (1,1) -- (0,1) -- (-1,0) -- (-1,-1);
\fill[thick, color=black]
(0,0) circle (2pt) (1,0) circle (2pt) (0,1) circle (2pt)
(1,1) circle (2pt) (-1,0) circle (2pt) (0,-1) circle (2pt)
(-1,-1) circle (2pt);
\draw[thick, color=red]
(0,0) circle (3pt);
\draw[thick,  color=black]
(0.0,-2.3) node{$H=A+B$};
\draw[thick,  color=black]
(1.3,0.5) node{$F_4$} (0.5,1.3) node{$F_5$} 
(-0.8,0.8) node{$F_6$} (-1.3,-0.5) node{$F_1$}
(-0.5,-1.3) node{$F_2$} (0.8,-0.8) node{$F_3$};
\end{tikzpicture}
\spaceA
\begin{tikzpicture}[scale=\scaleA]
\draw[color=oliwkowy!40] (-0.3,-1.3) grid (1.3,0.3);
\draw[thick, color=black]
(0,0) -- (1,0) -- (0,-1) -- (0,0);
\fill[pattern color=blue!50, pattern=north west lines]
(0,0) -- (1,0) -- (0,-1) -- (0,0);
\fill[thick, color=black]
(0,0) circle (2pt) (1,0) circle (2pt) (0,-1) circle (2pt);
\draw[thick, color=red]
(0,0) circle (3pt);
\draw[thick,  color=black]
(0.5,-1.7) node{$A$};
\end{tikzpicture}
\spaceA
\begin{tikzpicture}[scale=\scaleA]
\draw[color=oliwkowy!40] (-1.3,-0.3) grid (0.3,1.3);
\draw[thick, color=black]
(0,0) -- (-1,0) -- (0,1) -- (0,0);
\fill[pattern color=darkgreen!50, pattern=north west lines]
(0,0) -- (-1,0) -- (0,1) -- (0,0);
\fill[thick, color=black]
(0,0) circle (2pt) (-1,0) circle (2pt) (0,1) circle (2pt);
\draw[thick, color=red]
(0,0) circle (3pt);
\draw[thick,  color=black]
(-0.5,-0.7) node{$B$};
\end{tikzpicture}
$$
Here, the ray $\rho_i$ of the inner normal fan $\Sigma$ corresponds to the facet $F_i$ of $H$. Let {$D_i \coloneqq \overline{\orb(\rho_i)}$}. The Minkowski summands $A$ and $B$ of $H$ correspond to the divisors $D_A = D_2 + D_3 + D_4$ and $D_B = D_1 + D_5 + D_6$.
The interesting line bundle on $X$ arises from 
$\Delta^+=A$ and $\Delta^-=2B+(1,-1)$ with 
$D_{\Delta^-} = D_1 + D_2 + 2D_3 + D_4 + D_5.$
The figure below displays the set-theoretic difference 
$(2B+(1,-1))\setminus A$. It shows that 
$
\gH^1(X,\CO_X(\Delta^+ - \Delta^-)) \cong \oplus_{m \in M} \tH^0(\Delta^- \setminus (\Delta^+ + m))
$ 
is $2$-dimensional in degree $m = 0$, because
$
\dim \tH^0(\Delta^- \setminus \Delta^+) = \dim \tH^0((2B + (1,-1)) \setminus A) = 2.
$
$$
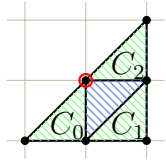
\newcommand{\scaleA}{0.8}
\newcommand{\spaceA}{\hspace*{5em}}
\begin{tikzpicture}[scale=\scaleA]
\draw[color=oliwkowy!40] (-2.3,-0.3) grid (0.3,2.3);
\draw[thick, color=black]
(0,0) -- (-2,0) -- (0,2) -- (0,0);
\fill[pattern color=green!50, pattern=north west lines]
(0,0) -- (-2,0) -- (0,2) -- (0,0);
\draw[thick, color=black]
(-1,0) -- (0,1) -- (-1,1) -- (-1,0);
\fill[pattern color=blue!50, pattern=north west lines]
(-1,0) -- (0,1) -- (-1,1) -- (-1,0);
\fill[thick, color=black]
(0,0) circle (2pt) (-1,0) circle (2pt) (0,1) circle (2pt)
(-1,1) circle (2pt) (-2,0) circle (2pt) (0,2) circle (2pt);
\draw[thick, color=red]
(-1,1) circle (3pt);
\draw[thick,  color=black]
(-1.3,0.25) node{$C_0$} (-0.3,0.25) node{$C_1$} (-0.3,1.25) node{$C_2$};
\end{tikzpicture}
$$
There is no other shift $m$ such that $(2B -m) \setminus A$, or equivalently $2B \setminus (A +m)$,
has non-trivial reduced $0$-th cohomology. Hence, $\gH^1(X, \CO_X(\Delta^+ - \Delta^-))$ is $2$-dimensional, sitting completely in degree $m = 0$.\\[1ex]
Denote the connected components of $\Delta^- \setminus \Delta^+$ by $C_0$, $C_1$, $C_2$ and set $\nabla_i \coloneqq C_i \cup \Delta^+$.
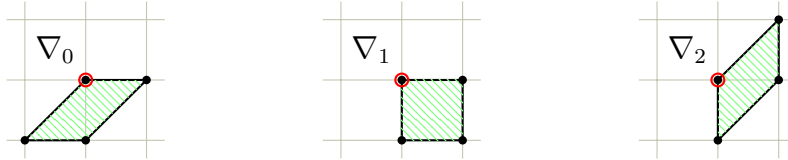
$$
\newcommand{\scaleA}{0.8}
\newcommand{\spaceA}{\hspace*{5em}}
\begin{tikzpicture}[scale=\scaleA]
\draw[color=oliwkowy!40] (-2.3,-0.3) grid (0.3,2.3);
\draw[thick,  color=black]
(-1.5,1.5) node{$\nabla_0$};
\draw[thick, color=black]
(-2,0) -- (-1,1) -- (0,1) -- (-1,0) -- (-2,0);
\fill[pattern color=green!50, pattern=north west lines]
(-2,0) -- (-1,1) -- (0,1) -- (-1,0) -- (-2,0);
\fill[thick, color=black]
(-1,0) circle (2pt) (0,1) circle (2pt)
(-1,1) circle (2pt) (-2,0) circle (2pt);
\draw[thick, color=red]
(-1,1) circle (3pt);
\end{tikzpicture}
\spaceA
\begin{tikzpicture}[scale=\scaleA]
\draw[color=oliwkowy!40] (-2.3,-0.3) grid (0.3,2.3);
\draw[thick,  color=black]
(-1.5,1.5) node{$\nabla_1$};
\draw[thick, color=black]
(0,0) -- (-1,0) -- (-1,1) -- (0,1) -- (0,0);
\fill[pattern color=green!50, pattern=north west lines]
(0,0) -- (-1,0) -- (-1,1) -- (0,1) -- (0,0);
\fill[thick, color=black]
(0,0) circle (2pt) (-1,0) circle (2pt) (0,1) circle (2pt)
(-1,1) circle (2pt);
\draw[thick, color=red]
(-1,1) circle (3pt);
\end{tikzpicture}
\spaceA
\begin{tikzpicture}[scale=\scaleA]
\draw[color=oliwkowy!40] (-2.3,-0.3) grid (0.3,2.3);
\draw[thick,  color=black]
(-1.5,1.5) node{$\nabla_2$};
\draw[thick, color=black]
(-1,0) -- (-1,1) -- (0,2) -- (0,1) -- (-1,0);
\fill[pattern color=green!50, pattern=north west lines]
(-1,0) -- (-1,1) -- (0,2) -- (0,1) -- (-1,0);
\fill[thick, color=black]
(-1,0) circle (2pt) (0,1) circle (2pt)
(-1,1) circle (2pt) (0,2) circle (2pt);
\draw[thick, color=red]
(-1,1) circle (3pt);
\end{tikzpicture}
$$
From the facet presentations of $\nabla_0$, $\nabla_1$ and $\nabla_2$ we can read off the associated divisors
$D_{\nabla_0} = D_1 + D_2 + D_3 + D_4$,
$D_{\nabla_1} = D_2 + 2D_3 + D_4$,
$D_{\nabla_2} = D_2 + D_3 + D_4 + D_5$.
In~this case, sequence~(\ref{equation:inducedSESMoreComponents}) in subsection~(\ref{subsubsection:manyComp}) is the sequence
\begin{equation}
0  \to \CO_X(\Delta^+)^{2}
\xrightarrow{\begin{pmatrix*} -1 & -1 \\ 1 & 0 \\ 0 & 1 \end{pmatrix*}} 
\oplus_{i = 0}^2\CO_X(\nabla_i) 
\xrightarrow{\begin{pmatrix*} 1 & 1 & 1 \end{pmatrix*}} \CO_X(\Delta^-) \to 0
\end{equation}
in $\gExt\big(\CO_X(\Delta^-), \CO_X(\Delta^+)^2\big)_0$. It corresponds to \mbox{$([C_1], [C_2]) \in \tH^0(\Delta^- \setminus \Delta^+)^2$}.\\[1ex]
The pushout construction described in Corollary \ref{cor:singleExtensions} yields two extension sequences 
$0 \to \CO_X(\Delta^+) \to \CG_1 \to \CO_X(\Delta^-) \to 0$ and
$0 \to \CO_X(\Delta^+) \to \CG_2 \to \CO_X(\Delta^-) \to 0$
in $\gExt\big(\CO_X(\Delta^-), \CO_X(\Delta^+)\big)_0$. 
The first is represented by $[C_1] \in \tH^0(\Delta^- \setminus \Delta^+)$ and the second by $[C_2] \in \tH^0(\Delta^- \setminus \Delta^+)$. The class of the constant cocycle on $\Delta^- \setminus \Delta^+$ is trivial, so \mbox{$[C_0] = -[C_1] - [C_2] \in \tH^0(\Delta^- \setminus \Delta^+)$} and the negative of the Baer sum of the two sequences above is represented by $[C_0] \in \tH^0(\Delta^- \setminus \Delta^+)$.

\subsection{General Position of Polyhedra}
\label{subsection:generalPositionOfPolyhedra}
In the previous section we dealt with the special case of one polyhedron $\Delta^+$ being contained in the other polyhedron $\Delta^-$.  In this section we deal with $\Delta^+$ and $\Delta^-$ lying in general position.

\subsubsection{Observing Two Problems}\label{subsubsection:introducingTwoProblems}
To illustrate the problem for the general case, we start with
an example. Let $\Sigma_0$ be the $2$-dimensional, singular fan made 
from the rays spanned by $(1,0)$, $(0,1)$, $(-2,-1)$, and $(0,-1)$, respectively (displayed in black below).
Adding the rays spanned by $(-1,0)$ and $(-1,-1)$ (displayed in blue below), we obtain the smooth subdivision
$\Sigma$.
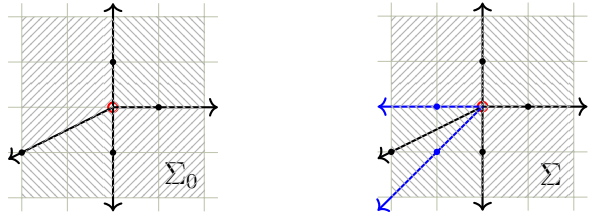
$$
\newcommand{\scaleA}{0.6}
\newcommand{\spaceA}{\hspace*{5em}}
\begin{tikzpicture}[scale=\scaleA]
\draw[color=oliwkowy!40] (-2.3,-2.3) grid (2.3,2.3);
\draw[thick, color=red]
(0,0) circle (3pt);
\draw[thick,  color=black]
(1.5,-1.5) node{$\Sigma_0$};
\draw[thick,  color=black, ->]
(0,0) -- (2.3,0);
\draw[thick,  color=black, ->]
(0,0) -- (0,-2.3);
\draw[thick,  color=black, ->]
(0,0) -- (0,2.3);
\draw[thick,  color=black, ->]
(0,0) -- (-2.3,-1.15);
\fill[pattern color=black!30, pattern=north east lines]
(0,0) -- (2,0) -- (2,-2) -- (0, -2);
\fill[pattern color=black!30, pattern=north west lines]
(0,0) -- (2,0) -- (2, 2) -- (0,2);
\fill[pattern color=black!30, pattern=north east lines]
(0,0) -- (-2,-1) -- (-2, 2) -- (0, 2);
\fill[pattern color=black!30, pattern=north west lines]
(0,0) -- (0,-2) -- (-2,-2) -- (-2, -1);
\fill[thick,  color=black]
(1,0) circle (2pt) (0,1) circle (2pt) (-2,-1) circle (2pt) 
(0,-1) circle (2pt);
\end{tikzpicture}
\spaceA
\begin{tikzpicture}[scale=\scaleA]
\draw[color=oliwkowy!40] (-2.3,-2.3) grid (2.3,2.3);
\draw[thick, color=red]
(0,0) circle (3pt);
\draw[thick,  color=black]
(1.5,-1.5) node{$\Sigma$};
\draw[thick,  color=black, ->]
(0,0) -- (2.3,0);
\draw[thick,  color=black, ->]
(0,0) -- (0,-2.3);
\draw[thick,  color=black, ->]
(0,0) -- (0,2.3);
\draw[thick,  color=black, ->]
(0,0) -- (-2.3,-1.15);
\draw[thick,  color=blue, ->]
(0,0) -- (-2.3,0);
\draw[thick,  color=blue, ->]
(0,0) -- (-2.3,-2.3);
\fill[pattern color=black!30, pattern=north east lines]
(0,0) -- (2,0) -- (2,-2) -- (0, -2);
\fill[pattern color=black!30, pattern=north west lines]
(0,0) -- (2,0) -- (2, 2) -- (0,2);
\fill[pattern color=black!30, pattern=north east lines]
(0,0) -- (-2,0) -- (-2, 2) -- (0, 2);
\fill[pattern color=black!30, pattern=north west lines]
(0,0) -- (-2,0) -- (-2, -1) -- (0, 0);
\fill[pattern color=black!30, pattern=north west lines]
(0,0) -- (-2, -1) -- (-2,-2) -- (0, 0);
\fill[pattern color=black!30, pattern=north west lines]
(0,0) -- (-2, -2) -- (0,-2) -- (0, 0);
\fill[thick,  color=black]
(1,0) circle (2pt) (0,1) circle (2pt) (-2,-1) circle (2pt) 
(0,-1) circle (2pt);
\fill[thick,  color=blue]
(-1,0) circle (2pt) (-1,-1) circle (2pt);
\end{tikzpicture}
$$
Denote $X:=\toric(\Sigma)$ and $f\colon X\to X_0:=\toric(\Sigma_0)$, the 
birational contraction. The Picard group $\Pic(X_0)$ is
freely generated by (the sheaves represented by) 
the polytopes $C$ and $D$ displayed below. $\Pic(X)=\Cl(X)$ 
has $\{A,B,C,D\}$ as a basis:
$$
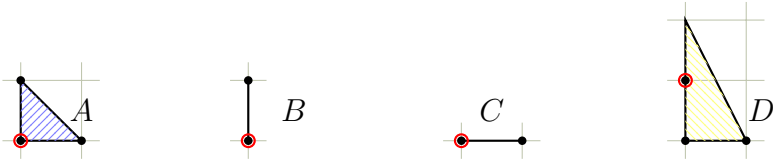
\newcommand{\scaleA}{0.8}
\newcommand{\spaceA}{\hspace*{4em}}
\begin{tikzpicture}[scale=\scaleA]
\draw[color=oliwkowy!40] (-0.3,-0.3) grid (1.3,1.3);
\draw[thick, color=black]
(0,0) -- (1,0) -- (0,1) -- (0,0);
\fill[pattern color=blue!50, pattern=north east lines]
(0,0) -- (1,0) -- (0,1) -- (0,0);
\fill[thick, color=black]
(0,0) circle (2pt) (1,0) circle (2pt) (0,1) circle (2pt);
\draw[thick, color=red]
(0,0) circle (3pt);
\draw[thick,  color=black]
(1,0.5) node{$A$};
\end{tikzpicture}
\spaceA
\begin{tikzpicture}[scale=\scaleA]
\draw[color=oliwkowy!40] (-0.3,-0.3) grid (0.3,1.3);
\draw[thick, color=black]
(0,0) -- (0,1);
\fill[thick, color=black]
(0,0) circle (2pt) (0,1) circle (2pt);
\draw[thick, color=red]
(0,0) circle (3pt);
\draw[thick,  color=black]
(0.75,0.5) node{$B$};
\end{tikzpicture}
\spaceA
\begin{tikzpicture}[scale=\scaleA]
\draw[color=oliwkowy!40] (-0.3,-0.3) grid (1.3,0.3);
\draw[thick, color=black]
(0,0) -- (1,0);
\fill[thick, color=black]
(0,0) circle (2pt) (1,0) circle (2pt);
\draw[thick, color=red]
(0,0) circle (3pt);
\draw[thick,  color=black]
(0.5,0.5) node{$C$};
\end{tikzpicture}
\spaceA
\begin{tikzpicture}[scale=\scaleA]
\draw[color=oliwkowy!40] (-0.3,-0.3) grid (1.3,2.3);
\draw[thick, color=black]
(0,0) -- (1,0) -- (0,2) -- (0,0);
\fill[pattern color=yellow!50, pattern=north west lines]
(0,0) -- (1,0) -- (0,2) -- (0,0);
\fill[thick, color=black]
(0,0) circle (2pt) (1,0) circle (2pt) (0,1) circle (2pt);
\draw[thick, color=red]
(0,1) circle (3pt);
\draw[thick,  color=black]
(1.25,0.5) node{$D$};
\end{tikzpicture}
$$
Define $\Delta^-:=D$ and $\Delta^+:=C$.
The difference $\Delta^-\setminus\Delta^+$
splits into two components:
$$
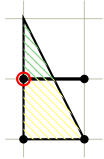
\newcommand{\scaleA}{0.8}
\newcommand{\spaceA}{\hspace*{4em}}
\begin{tikzpicture}[scale=\scaleA]
\draw[color=oliwkowy!40] (-0.3,-0.3) grid (1.3,2.3);
\draw[thick, color=black]
(0,0) -- (1,0) -- (0,2) -- (0,0);
\draw[thick, color=black]
(0,1.01) -- (1,1.01) (0,0.99) -- (1,0.99);
\fill[pattern color=yellow!50, pattern=north west lines]
(0,0) -- (1,0) -- (0.5,1) -- (0,1) -- (0,0);
\fill[pattern color=darkgreen!50, pattern=north west lines]
(0,1) -- (0.5,1) -- (0,2) -- (0,1);
\fill[thick, color=black]
(0,0) circle (2pt) (1,0) circle (2pt) (0,1) circle (2pt)
(1,1) circle (2pt);
\draw[thick, color=red]
(0,1) circle (3pt);
\end{tikzpicture}
$$
However, since $\Delta^+$ is not contained in $\Delta^-$, there is no
map from $\CO_X(\Delta^+)$ into the sheaves of the two components, which 
are subsheaves of $\CO_X(\Delta^-)$. The solution to this problem is to replace
$\Delta^+$ by $(\Delta^+\cap\Delta^-)$, then proceed as in
section~(\ref{subsection:inclusionOfPolyhedra}), and, finally, 
to use functoriality of $\gExt$ along the
embedding $\CO_X(\Delta^+\cap\Delta^-)\hookrightarrow\CO_X(\Delta^+)$.
\\[0.5ex]
But here one can spot the second problem. As in the example,
the intersection $\Delta^+\cap\Delta^-$ is not necessarily a lattice
polyhedron. This will be overcome by refining the lattice. We replace the lattice $M$ by a larger lattice $\widetilde{M} \supseteq M$ such that $\Delta^+ \cap \Delta^-$ is a lattice polyhedron with respect to $\widetilde{M}$.
Dually, this means to consider some sublattice $\widetilde{N} \subseteq N$ of finite index and the induced finite covering $p \colon \toric(\Sigma, \widetilde{N}) \to \toric(\Sigma, N)$.

\subsubsection{The Intersection is a Lattice Polyhedron}
\label{subsubsection:IntersectionLatticePolyhedron}
We start with the case where the intersection \mbox{$\Delta^+ \cap  \Delta^-$} is a lattice polyhedron with respect to the lattice $M$. 
Suppose that $\Delta^- \setminus \Delta^+ = \Delta^- \setminus (\Delta^+ \cap \Delta^-)$ consists of $n+1$ connected components $C_0, \dots, C_n$. 
Set $\nabla_i \coloneqq C_i \cup  (\Delta^+ \cap \Delta^-)$ for $i \in \{0, \dots, n\}$. 
As before, we can assume without loss of generality that the fan $\Sigma$ is refined enough so that $\Delta^+$, $\Delta^-$, $\Delta^+\cap \Delta^-$ and $\nabla_i$ for $i \in \{0, \dots, n\}$ are compatible with $\Sigma$ and define nef Cartier divisors on $X = \toric(\Sigma)$.\\[1ex]
Consider the following pushout diagram:
\begin{equation}\label{equation:pushoutDiagramOnX}
\xymatrix@R=4.5ex@C=1.7em{
	0 \ar[r] & 
	\CO_X(\Delta^+ \cap \Delta^-)^n \ar[r]^-{} \ar@{^(->}[d]^-{h^n} & 
	\oplus_{i = 0}^n \CO_X(\nabla_i) \ar[r] \ar[d] &
	\CO_X(\Delta^-) \ar[r] \ar@{=}[d]& 0
	\\
	0 \ar[r] & 
	\CO_X(\Delta^+)^n \ar[r] & 
	\CH \ar[r] &
	\CO_X(\Delta^-) \ar[r] & 0.
}
\end{equation}
Here the upper exact sequence is constructed as in subsection~(\ref{subsubsection:manyComp}). 
It is an extension sequence in 
$\gExt\big(\CO_X(\Delta^-), \CO_X(\Delta^+ \cap \Delta^-)^n\big)_0$. 
By Theorem \ref{theorem:ExtMultipleComponents} this corresponds to the $n$-tuple 
$([C_1], \dots, [C_n]) \in \tH^0(\Delta^- \setminus (\Delta^+ \cap \Delta^-))^n = \tH^0(\Delta^- \setminus \Delta^+)^n$.
The inclusion of lattice polyhedra $\Delta^+ \cap \Delta^- \hookrightarrow \Delta^+$ induces the embedding of sheaves $h \colon \CO_X(\Delta^+ \cap \Delta^-)  \hookrightarrow \CO_X(\Delta^+)$ and its $n$-th power $h^n \colon \CO_X(\Delta^+ \cap \Delta^-)^n \hookrightarrow \CO_X(\Delta^+)^n$.
The sheaf $\CH$ is the pushout  of the left square and the universal property of the pushout induces a map $\CH \to \CO_X(\Delta^-)$, that makes the lower sequence exact and the right square commutative. All together we obtain a map of complexes from the upper to the lower exact sequence. This is functoriality of $\gExt(\CO_X(\Delta^-), -)$.

\begin{proposition}
	\label{prop:intersectionLatticePolyhedron}
	Via the identification 
	$\gExt\big(\CO_X(\Delta^-), \CO_X(\Delta^+)^n\big)_0
	= \tH^0\big(\Delta^- \setminus \Delta^+\big)^n$
	the short exact extension sequence 
	\begin{equation}
	0 \to \CO_X(\Delta^+)^n \to \CH \to \CO_X(\Delta^-) \to 0
	\end{equation}
	corresponds to the $n$-tuple 
	$
	([C_1], \dots, [C_n]) \in \tH^0\big(\Delta^- \setminus \Delta^+\big)^n.
	$
\end{proposition}

\begin{proof}
	Tensoring both sequences in diagram~(\ref{equation:pushoutDiagramOnX}) with $\CO_X(\Delta^-)^{-1}$ yields
	\begin{equation}
	\xymatrix@R=4.5ex@C=1.7em{
		0 \ar[r]  &
		\CO_X((\Delta^+ \cap \Delta^-) - \Delta^-)^n \ar[r]^-{} \ar@{^(->}[d]^-{(h')^n} & 
		\oplus_{i = 0}^n \CO_X(\nabla_i - \Delta^-) \ar[r] \ar[d] &
		\CO_X \ar[r] \ar@{=}[d]& 0
		\\
		0 \ar[r] & 
		\CO_X(\Delta^+ - \Delta^-)^n \ar[r] & 
		\CH' \ar[r] &
		\CO_X \ar[r] & 0,
	}
	\end{equation}
	where $\CH' = \CH \otimes \CO_X(\Delta^-)^{-1}$ is also the pushout of the left square.
	The map of complexes between the short exact sequences induces a map of complexes between the long exact sequences in cohomology. In particular we obtain a commuting square
	\begin{equation}
	\xymatrix@R=4.5ex@C=1.7em{
		\Gamma(\CO_X) \ar[r]^{d \hspace{2cm}} \ar@{=}[d] & \gH^1\big(X, \CO_X(\Delta^+ - \Delta^-)\big)^n \ar[d]^{\gH^1((h')^n)}_{\cong} & \ni & \mu \ar@{|->}[d] \\
		\Gamma(\CO_X) \ar[r]^{d \hspace{2cm}}  & \gH^1\big(X, \CO_X((\Delta^+ \cap \Delta^- )- \Delta^-)\big)^n & \ni & \eta,
	}
	\end{equation}
	where $\eta$ and $\mu$ denote the images of $1 \in \Gamma\big(X, \CO_X\big)_0$ under the differential $d$. By Theorem \ref{theorem:ExtMultipleComponents} the isomorphism from $\gH^1\big(X, \CO_X(\Delta^+ - \Delta^-)\big)^n$ to reduced singular cohomology maps $\eta$ to 
	$
	([C_1], \dots, [C_n]) \in \tH^0\big(\Delta^- \setminus (\Delta^+ \cap \Delta^-)\big)^n = \tH^0\big(\Delta^- \setminus \Delta^+\big)^n.
	$
	It is easy to see that $\mu = \gH^1((h')^n)^{-1}(\eta)$ is also mapped to
	$
	([C_1], \dots, [C_n]) \in \tH^0\big(\Delta^- \setminus \Delta^+\big)^n.
	$
\end{proof}

\subsubsection{The Intersection is Not a Lattice Polyhedron}
\label{subsubsection:IntersectionNotLatticePolyhedron}
We now deal with the case where $\Delta^+ \cap  \Delta^- \subseteq M_\R$ is not a lattice polyhedron with respect to the lattice $M$. 
Take any refinement $\widetilde{M} \supseteq M$ such that $M$ is a sublattice of finite index in $\widetilde{M}$ and $\Delta^+ \cap  \Delta^- \subseteq \widetilde{M}_\R = M_\R \cong \R^\rankM$ is a lattice polyhedron with respect to $\widetilde{M}$.
Dually, $\widetilde{N} \subseteq N \cong \Z^\rankM$ is a sublattice of finite index.  
Let $G \coloneqq N / \widetilde{N}$ be the finite quotient group. 
For a toric variety $X = \toric(\Sigma,N)$ we can consider the fan $\Sigma$ in $N_\R = \widetilde{N}_\R$ with respect to the lattice $\widetilde{N}$ and obtain a second toric variety $\widetilde{X} = \toric(\Sigma,\widetilde{N})$ realizing $X$ as the geometric quotient $X = \widetilde{X} / G$. 
The lattice inclusion $\iota \colon \widetilde{N}  \hookrightarrow N$ induces the toric covering morphism $p \colon \widetilde{X} \rightarrow X$ \cite[Prop. 3.3.7]{CoxBook}. 
We pull the sheaves $\CO_X(\Delta^+)$ and $\CO_X(\Delta^-)$ back to $\widetilde{X}$ via $p$ and use the results from section~(\ref{subsubsection:IntersectionLatticePolyhedron}).

\begin{example}
	\label{example:scaleInEveryDirection}
	For $M = \Z^\rankM$ and its dual $N = \Z^\rankM$ take $\widetilde{M}$ to be $(\frac{1}{d}\Z)^\rankM$ and correspondingly $\widetilde{N} = (d\Z)^\rankM$. 
	Then each ray generator $\widetilde{v}_\rho \in \widetilde{N}$ of a ray $\rho \in \Sigma(1)$ is the $d$-multiple of the corresponding ray generator $v_\rho \in N$. 
	The coefficients of the pullback $p^*D$ of a Weil divisor $D$ on $X$ will be $d$-multiples of the coefficients of $D$.
\end{example}

\begin{construction}
\label{construction:pushforwardSequence}
	Back to our general case, the refinement $\widetilde{M} \supseteq M$ was chosen so that $\Delta^+ \cap \Delta^-$ is a lattice polyhedron with respect to $\widetilde{M}$ and we are in the case of section (\ref{subsubsection:IntersectionLatticePolyhedron}), however on the space $\widetilde{X}$.
	We have the short exact extension sequence
	\begin{equation}\label{equation:intersectoinExtensionOnXTilde}
	0 \to \CO_{\widetilde{X}}(\Delta^+ \cap \Delta^-)^n  \to \oplus_{i = 0}^n \CO_{\widetilde{X}}(\nabla_i) \to \CO_{\widetilde{X}}(\Delta^-) \to 0
	\end{equation}
	and the short exact extension sequence
	\begin{equation}\label{equation:extensionOnXTilde}
	0 \to \CO_{\widetilde{X}}(\Delta^+)^n  \to \widetilde{\CH} \to \CO_{\widetilde{X}}(\Delta^-) \to 0
	\end{equation}
	induced by the embedding $\CO_{\widetilde{X}}(\Delta^+ \cap \Delta^-)^n  \hookrightarrow \CO_{\widetilde{X}}(\Delta^+)^n$. 
	Under the identification with 
	$\tH^0(\Delta^- \setminus \Delta^+)^n$, both correspond to the $n$-tuple
	$
	([C_1], \dots, [C_n]) \in  \tH^0(\Delta^- \setminus \Delta^+)^n
	$
	by Theorem~\ref{theorem:ExtMultipleComponents} and Proposition~\ref{prop:intersectionLatticePolyhedron}.
	Since $p \colon \widetilde{X} \to X$ is affine, 
	the pushfoward $p_*$ is exact and we obtain a short exact extension sequence of sheaves on $X$:
	\begin{equation}\label{equation:pushforwardExtensionOnXTilde}
	0 \to p_*\CO_{\widetilde{X}}(\Delta^+)^n  \to p_*\widetilde{\CH} \to p_*\CO_{\widetilde{X}}(\Delta^-) \to 0.
	\end{equation}
	Recall that in the first step of the identification of $\gExt\big(\CO_{\widetilde{X}}(\Delta^-), \CO_{\widetilde{X}}(\Delta^+)^n\big)_0$ with $\tH^0\big(\Delta^- \setminus \Delta^+\big)^n$ we tensor sequence~(\ref{equation:extensionOnXTilde}) with $\CO_{\widetilde{X}}(\Delta^-)^{-1}$. Its pushforward is
	\begin{equation}
	0 
	\to p_*(\CO_{\widetilde{X}}(\Delta^+ - \Delta^-))^n
	\to p_*(\widetilde{\CH}')
	\to p_*\CO_{\widetilde{X}}
	\to 0.
	\end{equation}
	By the projection formula this
	is equal to sequence~(\ref{equation:pushforwardExtensionOnXTilde}) tensored with $\CO_X(\Delta^-)^{-1}$.
	For an affine morphism $p \colon \widetilde{X} \to X$ of noetherian separated schemes and a short exact sequence $0 \to \CF \to \CG \to \CH \to 0$
	of quasi-coherent sheaves on $\widetilde{X}$
	 there is a commuting diagram in which all vertical morphisms are isomorphisms:
	$$
	\xymatrix@R=4.5ex@C=1.7em{
		\Gamma(\widetilde{X}, \CF) \ar[r] \ar[d]^{\cong} & \Gamma(\widetilde{X}, \CG) \ar[r] \ar[d]^{\cong} &  \Gamma(\widetilde{X},\CH) \ar[r]^{\partial} \ar[d]^{\cong} & \gH^1(\widetilde{X},\CF) \ar[d]^{\cong}  \ar[r] \ar[d] & \gH^1(\widetilde{X}, \CG) \ar[r] \ar[d]^{\cong} & \dots \\
		\Gamma(X, p_*\CF) \ar[r] & \Gamma(X, p_*\CG) \ar[r] & \Gamma(X,p_*\CH) \ar[r]^{\partial} & \gH^1(X,p_*\CF) \ar[r] & \gH^1(X, p_*\CG) \ar[r] & \dots.
	}
	$$
	In particular,
	$
	\gH^1\big(X,p_*\CO_{\widetilde{X}}(\Delta^+ - \Delta^-)\big)^n = \gH^1\big(\widetilde{X},\CO_{\widetilde{X}}(\Delta^+ - \Delta^-)\big)^n
	$
	and the image of $1 \in \Gamma(X, p_*\CO_{\widetilde{X}})$ under $\partial$ is the desired $n$-tuple $([C_1], \dots, [C_n])$.
\end{construction}
While $p_*\CO_{\widetilde{X}}(\Delta^+) = p_*p^*\CO_X(\Delta^+)$ and $p_*\CO_{\widetilde{X}}(\Delta^-) = p_*p^*\CO_X(\Delta^-)$ are in general not equal to $\CO_{X}(\Delta^+)$ and $\CO_{X}(\Delta^-)$, respectively, the latter are a direct summand:
\begin{equation}
\CO_X(\Delta^\pm) = \bigoplus_{m \in M} \CO_X(\Delta^\pm)_m = \bigoplus_{m \in M} (p_*p^*\CO_X(\Delta^\pm))_m \subseteq \bigoplus_{\widetilde{m} \in \widetilde{M}} (p_*p^*\CO_X(\Delta^\pm))_{\widetilde{m}}.
\end{equation}
Furthermore, $\CO_X(\Delta^\pm)$ corresponds precisely to the $G$-invariants of $p_*p^*\CO_X(\Delta^\pm)$. 

\begin{example}
	For $\CO_X(\Delta^\pm)$ and the pullback $p^*\CO_X(\Delta^\pm) = \CO_{\widetilde{X}}(\Delta^\pm)$ we have
	\begin{equation}
	\CO_X(\Delta^{\pm})(U_\sigma) = \chi^{m_\sigma^{\pm}}\cdot \bC[\sigma^\vee \cap M] \subseteq \chi^{m_\sigma^{\pm}} \cdot \bC[\sigma^\vee \cap \widetilde{M}]  = p_*\CO_{\widetilde{X}}(\Delta^+)(U_\sigma),
	\end{equation}
	where $\{m_\sigma^{\pm}\}_{\sigma \in \Sigma}$ is the Cartier data of $D_{\Delta^{\pm}}$. All $m_\sigma^\pm$ are contained in $M \subseteq \widetilde{M}$. 
	Via $\bC[\sigma^\vee \cap M] \hookrightarrow \bC[\sigma^\vee \cap \widetilde{M}]$ we can view the above as an inclusion of \mbox{$\bC[\sigma^\vee \cap M]$-modules}. 
	Any $M$-graded module obtains an $\widetilde{M}$-grading by adding zeros in degrees $\widetilde{m} \in \widetilde{M} \setminus M$. 
\end{example}
For the sheaves in the short exact sequence~(\ref{equation:pushforwardExtensionOnXTilde})
we have the $\widetilde{M}$-graded subsheaves $\CO_X(\Delta^+)$ of $p_*\CO_{\widetilde{X}}(\Delta^+)$ and $\CO_{X}(\Delta^-)$ of $p_*\CO_{\widetilde{X}}(\Delta^-)$. 
The inclusions are isomorphisms when restricted to $M \subseteq \widetilde{M}$. 
Because the morphisms in sequence~(\ref{equation:pushforwardExtensionOnXTilde}) are homogeneous of degree $0$, taking the $M$-graded part of $p_*\widetilde{\CH}$ yields a subsheaf 
$
\CH \coloneqq \oplus_{m \in M} (p_*\widetilde{\CH})_m \subseteq p_*\widetilde{\CH}
$ 
that fits into a commuting diagram of exact sequences
\begin{equation}\label{diagram:commutingDiagramSES}
\xymatrix@R=4.5ex@C=1.7em{
	0 \ar[r] & \CO_{X}(\Delta^+)^n  \ar[r] \ar@{^(->}[d]  &  \CH \ar[r] \ar@{^(->}[d]  & \CO_{X}(\Delta^-) \ar[r] \ar@{^(->}[d] & 0\\
	0 \ar[r] & p_*\CO_{\widetilde{X}}(\Delta^+)^n  \ar[r] &  p_*\widetilde{\CH} \ar[r] & p_*\CO_{\widetilde{X}}(\Delta^-) \ar[r] & 0.
}
\end{equation}
\begin{proposition}
	\label{prop:subPushforwardSequence}
	Given the commuting diagram~(\ref{diagram:commutingDiagramSES}), 
	the upper sequence
	\begin{equation}\label{equation:extensionSequenceGeneralCase}
	0 \to \CO_{X}(\Delta^+)^n \to \CH \to \CO_X(\Delta^-) \to 0
	\end{equation}
	 corresponds to the $n$-tuple
	$([C_1], \dots, [C_n]) \in \tH^0(\Delta^- \setminus \Delta^+)^n$
	under the identification
	$\gExt\big(\CO_X(\Delta^-), \CO_X(\Delta^+)^n\big)_0 = \tH^0(\Delta^- \setminus \Delta^+)^n.$
\end{proposition}
\begin{proof}
	Tensoring diagram~(\ref{diagram:commutingDiagramSES}) with $\CO_X(\Delta^-)^{-1}$ yields the  commutative diagram 
	\begin{equation}
	\xymatrix@R=4.5ex@C=1.7em{
		0 \ar[r] & \CO_{X}(\Delta^+- \Delta^-)^n  \ar[r] \ar@{^(->}[d]  &  \CH \otimes_{\CO_X} \CO_{X}(\Delta^-)^{-1} \ar[r] \ar@{^(->}[d]  & \CO_{X} \ar[r] \ar@{^(->}[d] & 0\\
		0 \ar[r] & p_*(\CO_{\widetilde{X}}(\Delta^+ - \Delta^-))^n  \ar[r] &  p_*\widetilde{\CH}' \ar[r] & p_*\CO_{\widetilde{X}} \ar[r] & 0,
	}
	\end{equation}
	the lower sequence inducing the desired $n$-tuple by Construction~\ref{construction:pushforwardSequence}.
	We obtain a commutative diagram of long exact sequences
	\begin{equation}
	\xymatrix@R=4.5ex@C=1.7em{
		\cdots \ar[r] &  \Gamma(X,\CO_{X}) \ar[r] \ar[d]^{(1)} & \gH^1\big(X,\CO_{X}(\Delta^+- \Delta^-)\big)^n \ar[r] \ar[d]^{(2)} &  \cdots   \\
		\cdots \ar[r] & \Gamma(X,p_*\CO_{\widetilde{X}}) \ar[r] & \gH^1\big(X,p_*(\CO_{\widetilde{X}}(\Delta^+ - \Delta^-))\big)^n \ar[r] & \cdots
	}
	\end{equation}
	The homomorphisms $(1)$ and $(2)$ restrict to isomorphisms in degree $0 \in M$, so \mbox{$1 \in \Gamma(X,\CO_X)_0 = \Gamma(X,p_*\CO_{\widetilde{X}})_0$} also maps to
	$([C_1], \dots, [C_n]) \in \tH^0(\Delta^- \setminus \Delta^+)^n.$
\end{proof}

\begin{corollary}
	\label{cor:pushoutSingleExtension}
	Sequence~(\ref{equation:extensionSequenceGeneralCase})
	induces $n$ extensions in $\gExt\big(\CO_X(\Delta^-),\CO_X(\Delta^+)\big)_0$,
	one for each $i = 1, \dots, n$:
	\begin{equation}
	\xymatrix@R=4.5ex@C=1.7em{
		0 \ar[r] & 
		\CO_X(\Delta^+)^n \ar[r] \ar@{->>}[d]^{pr_i} & 
		\CH \ar[r] \ar@{->}[d] &
		\CO_X(\Delta^-) \ar[r] \ar@{=}[d]& 0
		\\
		0 \ar[r] & 
		\CO_X(\Delta^+) \ar[r] & \CH_i \ar[r] &
		\CO_X(\Delta^-) \ar[r] & 0.
	}
	\end{equation}
	Via the identification
	$\gExt\big(\CO_X(\Delta^-), \CO_X(\Delta^+)\big)_0 =
	\tH^{0}\big(\Delta^- \setminus\Delta^+\big),$
	the $i$-th extension 
	corresponds to the class $[C_i] \in \tH^0\big(\Delta^- \setminus\Delta^+\big)$.
	In particular, the $n$ extensions for $i \in \{1, \dots, n\}$ form a basis of $\gExt\big(\CO_X(\Delta^-),\CO_X(\Delta^+)\big)_0$.\\[1ex]
	Furthermore, sequence~(\ref{equation:extensionSequenceGeneralCase}) is a universal extension for $\gExt\big(\CO_X(\Delta^-), \CO_X(\Delta^+)\big)_0$.
\end{corollary}

\section{Using Klyachko's Description of Toric Reflexive Sheaves}
\label{section:klyachko}
We briefly show how to construct the universal extension sequence in the case where $\Delta^+ \cap \Delta^-$ is not a lattice polyhedron in terms of Klyachko's description of toric reflexive sheaves (see \cite{klyachko} and
\cite{klyachkoICM}; a short summary
can be found in~\cite{payne}; see \cite{parliaments}
and \cite{tropical} for more recent approaches).
\subsection{Describing Sheaves via Filtrations}
\label{subsection:sheavesViaFiltrations}
Consider a 
toric variety $X=\toric(\Sigma)$ given by a fan $\Sigma$ in $N_\R$.
Let $1\in T\subseteq X$ denote the neutral element. 
Each $\CO_X$-module $\CE$ gives 
rise to a $\C$-vector space
$
E:=\CE(1):=\CE_1/\idm_{X,1}\CE_1,
$
where $\CE_1$ denotes the stalk of $\CE$ at $1\in X$ and 
$\idm_{X,1}$ the maximal ideal of $1$. If $\CE$ is a $T$-equivariant, 
torsion free sheaf on $X$, 
the sections of $\CE$ on the open, affine, $T$-invariant 
subsets $\toric(\sigma)\subseteq X$ with $\sigma\in\Sigma$
are $M$-graded subsets of $E\otimes_\C\C[M]$. 
If, in addition, $\CE$ is reflexive, then $\CE$ is already determined 
by its restriction to open subsets whose complements are of codimension 
equal or greater than two. 
Via Klyachko's description~\cite{klyachko}, a toric reflexive sheaf $\CE$ corresponds 
to a set of decreasing $\Z$-filtrations
\begin{equation}
F(\CE)_\rho^\kbb \;=\; [\ldots \supseteq E_\rho^{\ell-1}\supseteq E_\rho^{\ell}\supseteq 
E_\rho^{\ell+1}\supseteq\ldots]\hspace{1em}(\ell\in\Z)
\end{equation}
of the vector space $E$ which are parametrized by the rays $\rho\in\Sigma(1)$. 
Let $v_\rho$ denote the primitive generator of the ray $\rho$.
The filtrations encode the sections of $\CE$ on the $T$-invariant 
open subsets $U_\rho=\toric(\rho)\subseteq X$ defined by $\rho$. 
Namely, for $u \in M$,
\begin{equation}
e\otimes\chi^u\in\Gamma(U_\rho,\CE)\quad\iff\quad 
e\in E_\rho^{-\langle u,v_\rho\rangle}=F(\CE)_\rho^{-\langle u,v_\rho\rangle}.
\end{equation}

\begin{remark}
\label{rem:klyachkoCompatability}
The reflexive sheaf $\CE$ defines a toric {\em vector bundle} if it is 
subject to \emph{Klyachko's compatibility condition}~\cite{klyachko}: 
For each cone $\sigma\in\Sigma$ there exists a decomposition 
$E =\bigoplus_{[u]\in M/M\cap\sigma^\perp}E_{[u]}$ so that 
$E^l_\rho = \sum_{\langle u,v_\rho\rangle\ge l} E_{[u]}$ for each 
$\rho \in \sigma(1)$.
\end{remark}

\subsection{Line and Tangent Bundles}
\label{subsection:examplesKlyachko}
Line bundles and the tangent are examples
of toric bundles with filtrations as follows 
(compare \cite[Example 2.3]{klyachko}).
\begin{example}\label{example:lineBundleKlyachko}
Let $D_\rho=\ko{\orb(\rho)}$ be the closure of the orbit defined by $\rho \in \Sigma(1)$. For $D=\sum_{\rho \in \Sigma(1)} \lambda_\rho\,D_\rho, \, \lambda_\rho \in \bZ$, the invertible sheaf
$\CO_{X}(D)$ is encoded by
\begin{equation}
E_\rho^\ell:=\left\{\begin{array}{cl}
\bC & \mbox{if } \ell \leq\lambda_\rho
\\
0 & \mbox{if } \ell \geq \lambda_\rho+1
\end{array}\right\}
\subseteq \bC \eqqcolon E.
\end{equation}
\end{example}

\begin{example}
	\label{subsubsection:tangentBundleKlyachko}
The tangent sheaf $\CT_{X}$ corresponds to the filtration
\begin{equation}
T^{\ell}_\rho:=\left\{\begin{array}{cl}
N_\bC = N \otimes_\bZ \bC & \mbox{if } \ell \leq 0
\\
\spann(\rho) & \mbox{if } \ell =1
\\
0 & \mbox{if } \ell \geq 2.
\end{array}\right\}\subseteq N_\bC \eqqcolon E.
\end{equation}
\end{example}
\begin{remark}
	Projective $n$-space $\PP^n$ is the toric variety associated to the normal fan $\normal(\Delta_n)$ of the standard $n$-simplex $\Delta_n$. The fan $\normal(\Delta_n)$ has $n+1$ rays $\rho_0, \dots, \rho_n$. 
	The direct sum of invertible sheaves $\oplus_{j = 0}^n \CO_{\PP^n}(\overline{\orb(\rho_j)})$ corresponds to the filtrations
	\begin{equation}
	E_{\rho_i}^\ell \coloneqq \left\{\begin{array}{cl}
	E & \mbox{if } \ell \leq 0 \\
	\bC \cdot e^{\rho_i} & \mbox{if } \ell = 1
	\\
	0 & \mbox{if } \ell \geq 2
	\end{array}\right\}
	\subseteq E \coloneqq \oplus_{j = 0}^n \bC \cdot e^{\rho_j} .
	\end{equation}
	The canonical surjection $\pi \colon E \twoheadrightarrow N_\bC, e^{\rho_i} \mapsto v_{\rho_i}$, where $v_{\rho_i}$ generates the ray $\rho_i$, induces the filtrations of $N_\bC$ corresponding to $\CT_{\PP^n}$. On $\ker(\pi) \cong \bC$ the induced filtrations are those of the structure sheaf $\CO_{\PP^n}$. This yields the \emph{Euler sequence}
	\begin{equation}
	0 \to \CO_{\PP^n} \to \oplus_{j = 0}^n \CO_{\PP^n}(\overline{\orb(\rho_j)}) \to \CT_{\PP^n} \to 0.
	\end{equation}
\end{remark}

\subsection{Pullbacks under Toric Morphisms}
\label{subsection:pullbacks}
Let $\CE$ be a toric reflexive sheaf on \mbox{$X \coloneqq \toric(\Sigma, N)$}. 
Similarly to subsection~(\ref{subsubsection:IntersectionNotLatticePolyhedron}), let $\widetilde{N} \subseteq N$ be a sublattice of finite index and $\widetilde{X} \coloneqq \toric(\Sigma, \widetilde{N})$.
Let $v_\rho \in N$ denote the primitive generator of the ray $\rho \in \Sigma(1)$ in $N$ and let $\overline{v_\rho}$ be its image in the quotient group $G \coloneqq N/\widetilde{N}$. For $d_\rho \coloneqq \min \{d \ge 1 \kst d \cdot \overline{v_\rho} = 0 \text{ in } G\}$ the order of $\overline{v_\rho}$ in $G$,
the multiple $\widetilde{v}_\rho \coloneqq d_\rho \cdot v_\rho \in \widetilde{N}$ is the primitive generator of $\rho$ in $\widetilde{N}$. Let $p \colon \widetilde{X} \to X$ be the toric covering morphism.\\
Analogously to \cite[Prop. 4.9]{Payne2006ToricVB} one can give the following description of the filtrations of a pullback of a toric reflexive sheaf on $X$ along $p$.
\begin{proposition}
	\label{prop:pullbackFiltration}
	Suppose the toric reflexive sheaf $\CE$ corresponds to the vector space $E = \CE(1)$ with filtrations $(E^l_\rho)_{l \in \Z}$ for $\rho \in \Sigma(1)$.
	Then the pullback $\CF \coloneqq p^*\CE$ of $\CE$ is a toric reflexive sheaf  on $\widetilde{X}$. It corresponds to the same vector space $F = E$ and the filtration for $\rho \in \Sigma(1)$ is given by 
	$
	F^l_\rho = E^{\lceil \frac{l}{d_\rho} \rceil}_\rho, \, l \in \Z,
	$
	where $\lceil \cdot \rceil$ denotes the ceiling function, giving the smallest integer equal to or larger than the argument. 
\end{proposition}
This filtration $(F_\rho^l)_{l \in \bZ}$ of $F = E$ can be thought of as the filtration $(E^l_\rho)_{l \in \Z}$ stretched by the factor $d_\rho$ and we will refer to it as the \emph{the $d_\rho$-th stretching of the filtration}.

\subsection{Constructing the Universal Extension using Klyachko's Description}
\label{subsection:constructionKlyachko}
We are in the setting of subsection~(\ref{subsubsection:IntersectionNotLatticePolyhedron}{)}.
Consider the two line bundles $\CO_X(\Delta^+)$ and $\CO_X(\Delta^-)$ on $X$. 
In Klyachko's description, let them be given by the $\bC$-vector spaces with decreasing $\Z$-filtrations $(E_+, E_{+, \rho}^\kbb)_{\rho \in \Sigma(1)}$ and $(E_-, E_{-, \rho}^\kbb)_{\rho \in \Sigma(1)}$, respectively. 
By Proposition~\ref{prop:pullbackFiltration} the pullback line bundles \mbox{$\CO_{\widetilde{X}}(\Delta^+) = p^*\CO_X(\Delta^+)$} and \mbox{$\CO_{\widetilde{X}}(\Delta^-) = p^*\CO_X(\Delta^-)$} on $\widetilde{X}$ correspond to the $\bC$-vector spaces and decreasing $\Z$-filtrations $(F_+, F_{+, \rho}^\kbb)_{\rho \in \Sigma(1)}$ and $(F_-, F_{-, \rho}^\kbb)_{\rho \in \Sigma(1)}$ with
$
F_{\pm} = E_{\pm}$ and $F_{\pm,\rho}^l = E_{\pm, \rho}^{\lceil \frac{l}{d_\rho} \rceil}.
$
The filtration $F_\rho^\kbb$ for a ray $\rho$ corresponding to a pullback sheaf $\CF = p^*\CE$ of $\CE$ on $X$ is a \emph{$d_\rho$-th stretching} in the sense that a proper inclusion can only occur every $d_\rho$ steps in the filtration:
\begin{equation}
\dots = \underbrace{F_\rho^{d_\rho k}}_{=E_\rho^{k}} \supseteq \underbrace{F_\rho^{d_\rho k + 1}}_{=E_\rho^{k+1}} = \dots = \underbrace{F_\rho^{d_\rho (k+1)}}_{=E_\rho^{k+1}} \supseteq \underbrace{F_\rho^{d_\rho (k+1) + 1}}_{=E_\rho^{k+2}} = \dots
\end{equation}
In sequence~(\ref{equation:extensionOnXTilde})
the outer two sheaves are the pullbacks of line bundles on $X$.
\begin{lemma}
	\label{lemma:middleSheafReflexiveandPower}
	The sheaf $\widetilde{\CH}$ in sequence~(\ref{equation:extensionOnXTilde}) is a reflexive sheaf on $\widetilde{X}$ and corresponds to a $\bC$-vector space with filtrations ${(\widetilde{H}, {\widetilde{H}_\rho}^\kbb)}_{\rho \in \Sigma(1)}$, where the filtration ${\widetilde{H}_\rho}^\kbb$ for the ray $\rho \in \Sigma(1)$ is a $d_\rho$-th stretching in the sense introduced above.
\end{lemma}
\begin{proof}
	Applying the contravariant functor $\sheafhom(-, \CO_{\widetilde{X}}) \colon \Coh(\widetilde{X}) \rightarrow \Coh(\widetilde{X})$ to the sequence~(\ref{equation:extensionOnXTilde}) twice yields a short exact sequence of double duals  (using that the outer two sheaves are locally free) with a canonical homomorphism from the original sequence. Reflexivity of the outer two sheaves and the five lemma give reflexivity of the sheaf $\widetilde{\CH}$.
	Let ${\widetilde{H}_\rho}^\kbb$, $F_{+, \rho}^\kbb$ and $F_{-, \rho}^\kbb$, $\rho \in \Sigma(1)$, denote the filtrations corresponding to the reflexive sheaves $\widetilde{\CH}$, $\CO_{\widetilde{X}}(\Delta^+)$ and $\CO_{\widetilde{X}}(\Delta^-)$, respectively.
	Since sequence~(\ref{equation:extensionOnXTilde}) is $T$-equivariant, it induces a short exact sequence of filtrations $0 \to F_{+, \rho}^\kbb \to {\widetilde{H}_\rho}^\kbb \to F_{-, \rho}^\kbb \to 0$ for each $\rho \in \Sigma(1)$. 
	Hence, the filtration $H_\rho^\kbb$ of $H$ is determined by the filtrations $F_{+,\rho}^\kbb$ and $F_{-, \rho}^\kbb$ and is thus also $d_\rho$-th stretching.
\end{proof}
\begin{remark}
	A filtration of a vector space that is a $d$-th stretching can also be \emph{squished} back: For a vector space $F$ with filtration
	\begin{equation}
	\dots = F^{d k} \supseteq F^{d k + 1}= \dots = F^{d (k+1)} \supseteq F^{d (k+1) + 1} = \dots
	\end{equation}
	the \emph{$d$-th squishing} of $(F, F^\kbb)$ is the vector space $F$ with filtration
	\begin{equation}
	\dots \supseteq F^{d (k-1)} \supseteq F^{d k} \supseteq F^{d(k+1)} \supseteq \dots
	\end{equation}
\end{remark}
\begin{corollary}
	\label{cor:rootSheaf}
	The reflexive sheaf $\CH$ on $X$ corresponding to $(H, H_\rho^\kbb)$, where $H  \coloneqq \widetilde{H}$ and the filtration $H_\rho^\kbb$ given by 
	$
	H_\rho^l \coloneqq {\widetilde{H}_\rho}^{d_\rho l} \big(= {\widetilde{H}_\rho}^{d_\rho l -1} = \dots = {\widetilde{H}_\rho}^{d_\rho l - (d_\rho -1)}\big), l \in \Z,
	$
	is the $d_\rho$-th squishing of ${\widetilde{H}_\rho}^\kbb$, $\rho \in \Sigma(1)$, fits into a short exact sequence
	\begin{equation}\label{equation:squishedSequence}
	0 \to \CO_X(\Delta^+)^n \to \CH \to \CO_X(\Delta^-) \to 0
	\end{equation}
	that pulls back to the short exact sequence~(\ref{equation:extensionOnXTilde}) under $p \colon \widetilde{X} \to X$.\\
	Sequence~(\ref{equation:squishedSequence}) is precisely the universal extension sequence~(\ref{equation:extensionSequenceGeneralCase}).
\end{corollary}

\begin{example}
	We continue with the example introduced in subsection~(\ref{subsubsection:introducingTwoProblems}) using Klyachko's language.
	Let $M = \bZ \oplus \bZ$, $N = \bZ \oplus \bZ$ and 
	recall the smooth fan $\Sigma$ in $N_\bR \cong \bR^2$ and the lattice polytopes $\Delta^+$ and $\Delta^-$ with respect to $M = \bZ \oplus \bZ$:
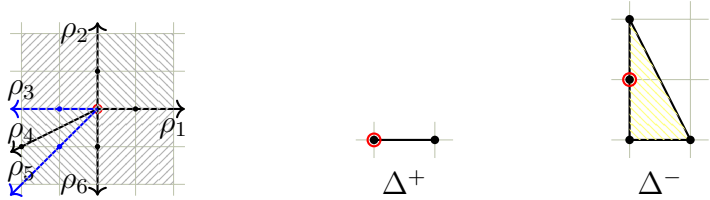
	$$
	\newcommand{\scaleA}{0.5}
	\newcommand{\spaceA}{\hspace*{5em}}
	\begin{tikzpicture}[scale=\scaleA]
	\draw[color=oliwkowy!40] (-2.3,-2.3) grid (2.3,2.3);
	\draw[thick, color=red]
	(0,0) circle (3pt);
	\draw[thick,  color=black, ->]
	(0,0) -- (2.3,0);
	\draw[thick,  color=black, ->]
	(0,0) -- (0,-2.3);
	\draw[thick,  color=black, ->]
	(0,0) -- (0,2.3);
	\draw[thick,  color=black, ->]
	(0,0) -- (-2.3,-1.15);
	\draw[thick,  color=blue, ->]
	(0,0) -- (-2.3,0);
	\draw[thick,  color=blue, ->]
	(0,0) -- (-2.3,-2.3);
	\fill[pattern color=black!30, pattern=north east lines]
	(0,0) -- (2,0) -- (2,-2) -- (0, -2);
	\fill[pattern color=black!30, pattern=north west lines]
	(0,0) -- (2,0) -- (2, 2) -- (0,2);
	\fill[pattern color=black!30, pattern=north east lines]
	(0,0) -- (-2,0) -- (-2, 2) -- (0, 2);
	\fill[pattern color=black!30, pattern=north west lines]
	(0,0) -- (-2,0) -- (-2, -1) -- (0, 0);
	\fill[pattern color=black!30, pattern=north west lines]
	(0,0) -- (-2, -1) -- (-2,-2) -- (0, 0);
	\fill[pattern color=black!30, pattern=north west lines]
	(0,0) -- (-2, -2) -- (0,-2) -- (0, 0);
	\fill[thick,  color=black]
	(1,0) circle (2pt) (0,1) circle (2pt) (-2,-1) circle (2pt) 
	(0,-1) circle (2pt);
	\fill[thick,  color=blue]
	(-1,0) circle (2pt) (-1,-1) circle (2pt);
	\draw[thick,  color=black]
	(2,-0.5) node{$\rho_1$};
	\draw[thick,  color=black]
	(-0.6,2) node{$\rho_2$};
	\draw[thick,  color=black]
	(-2,0.5) node{$\rho_3$};
	\draw[thick,  color=black]
	(-2,-0.65) node{$\rho_4$};
	\draw[thick,  color=black]
	(-2,-1.65) node{$\rho_5$};
	\draw[thick,  color=black]
	(-0.6,-2) node{$\rho_6$};
	\end{tikzpicture}
	\spaceA
	\begin{tikzpicture}[scale=0.8]
	\draw[color=oliwkowy!40] (-0.3,-0.3) grid (1.3,0.3);
	\draw[thick, color=black]
	(0,0) -- (1,0);
	\fill[thick, color=black]
	(0,0) circle (2pt) (1,0) circle (2pt);
	\draw[thick, color=red]
	(0,0) circle (3pt);
	\draw[thick,  color=black]
	(0.5,-0.7) node{$\Delta^+$};
	\end{tikzpicture}
	\spaceA
	\begin{tikzpicture}[scale=0.8]
	\draw[color=oliwkowy!40] (-0.3,-0.3) grid (1.3,2.3);
	\draw[thick, color=black]
	(0,0) -- (1,0) -- (0,2) -- (0,0);
	\fill[pattern color=yellow!50, pattern=north west lines]
	(0,0) -- (1,0) -- (0,2) -- (0,0);
	\fill[thick, color=black]
	(0,0) circle (2pt) (1,0) circle (2pt) (0,1) circle (2pt) (0, 2) circle (2pt);
	\draw[thick, color=red]
	(0,1) circle (3pt);
	\draw[thick,  color=black]
	(0.5,-0.7) node{$\Delta^-$};
	\end{tikzpicture}
	$$
	with $\Delta^+$ not contained in $\Delta^-$.
	The intersection $\Delta^+ \cap \Delta^-$ is not a lattice polytope with respect to $\bZ \oplus \bZ$. Consider the sublattice $\widetilde{N} = 2\bZ \oplus \bZ \hookrightarrow \bZ \oplus \bZ = N$ and dually $M = \bZ \oplus \bZ \hookrightarrow (\frac{1}{2}\bZ) \oplus \bZ = \widetilde{M}$, so that $\Delta^+ \cap \Delta^-$ is a lattice polytope with respect to $\widetilde{M}$. 
	Denote the unions of the components of $\Delta^- \setminus (\Delta^+ \cap \Delta^-)$ with $\Delta^+ \cap \Delta^-$ by 
	$$
	\nabla_0 \coloneqq 
	\newcommand{\scaleA}{0.8}
	\newcommand{\spaceA}{\hspace*{4em}}
	\raisebox{-1em}{\begin{tikzpicture}[scale=\scaleA]
	\draw[color=oliwkowy!40] (-0.3,0.7) grid (1.3,2.3);
	\draw[color=oliwkowy!40] (0.5,2.3) -- (0.5, 0.7);
	\draw[thick, color=black]
	(0,1) -- (0.5,1) -- (0,2) -- (0,1);
	\fill[pattern color=darkgreen!50, pattern=north west lines]
	(0,1) -- (0.5,1) -- (0,2) -- (0,1);
	\fill[thick, color=black]
	(0,1) circle (2pt) (0.5,1) circle (2pt)
	(0,2) circle (2pt);
	\draw[thick, color=red]
	(0,1) circle (3pt);
	\end{tikzpicture}}
	\text{ and }
	\nabla_1 \coloneqq 
	\raisebox{-1em}{\begin{tikzpicture}[scale=\scaleA]
	\draw[color=oliwkowy!40] (-0.3,-0.3) grid (1.3,1.3);
		\draw[color=oliwkowy!40] (0.5,1.3) -- (0.5, -0.3);
	\draw[thick, color=black]
	(0,0) -- (1,0) -- (0.5,1) -- (0,1) -- (0,0);
	\fill[pattern color=yellow!50, pattern=north west lines]
	(0,0) -- (1,0) -- (0.5,1) -- (0,1) -- (0,0);
	\fill[thick, color=black]
	(0,0) circle (2pt) (1,0) circle (2pt) (0,1) circle (2pt)
	(0.5,1) circle (2pt);
	\draw[thick, color=red]
	(0,1) circle (3pt);
	\end{tikzpicture}}.$$
	Set $\widetilde{X} \coloneqq \toric(\Sigma, \widetilde{N})$ and denote the covering morphism by $p \colon \widetilde{X} \to X$.
	The minimal ray generators $\widetilde{v}_{\rho_i}$, $i = 1, \dots, 6$, in the lattice $\widetilde{N}$ are
	$\widetilde{v}_{\rho_1} = (2,0)$, $\widetilde{v}_{\rho_2} = (0,1)$, $\widetilde{v}_{\rho_3} = (-2,0)$, $\widetilde{v}_{\rho_4} = (-2,-1)$, $\widetilde{v}_{\rho_5} = (-2,-2)$, and $\widetilde{v}_{\rho_6} = (0,-1)$.
	From these we determine the filtrations of the divisors $\widetilde{D}_{\Delta^+ \cap \Delta^-} $, $\widetilde{D}_{\Delta^+}$, $\widetilde{D}_{\Delta^-}$, $\widetilde{D}_{\nabla_0}$ and $\widetilde{D}_{\nabla_1}$ on $\widetilde{X}$.
	Taking pushouts allows us to calculate the filtrations for $\widetilde{\CH}$ on $\widetilde{X}$.
	In order to obtain the filtrations for the sheaf $\CH$ we need to take the $d_\rho$-squishings of the filtrations for $\widetilde{\CH}$, whenever $ d_\rho \ne 1$. This is the case for $\rho_1, \rho_3$ and $\rho_5$, with $d_\rho = 2$ in each case. 
	The following table depicts the resulting filtrations for $\CH$.
	\begin{center}
		\begin{tabular}{| C{0.6cm} |c c c c c c c c c | }
			\thickhline
			&&&&&&&&& \\ [-2ex]
			$\CH$ && \color{gray} 0 && \color{gray} 1 && \color{gray} 2 && \color{gray} 3 & \\
			\thickhline
			$\rho_1$ & $\supseteq$ & $(\bC_+ \oplus \bC_0 \oplus \bC_1)/\bC$ & $\supseteq$ & $0$ & $\supseteq$ & $0$ & $\supseteq$ & $0$ & $\supseteq$ \\
			\hline
			$\rho_2$ & $\supseteq$ & $(\bC_+ \oplus \bC_0 \oplus \bC_1)/\bC$ & $\supseteq$& $\bC_1$ & $\supseteq$ & $0$ & $\supseteq$ & $0$& $\supseteq$  \\
			\hline
			$\rho_3$ & $\supseteq$ & $(\bC_+ \oplus \bC_0 \oplus \bC_1)/\bC$ & $\supseteq$ & $(\bC_+ \oplus \bC_0 \oplus \bC_1)/\bC$ & $\supseteq$ & $0$ & $\supseteq$ & $0$ & $\supseteq$ \\
			\hline
			$\rho_4$ & $\supseteq$ & $(\bC_+ \oplus \bC_0 \oplus \bC_1)/\bC$ & $\supseteq$ & $(\bC_+ \oplus \bC_0 \oplus \bC_1)/\bC$ & $\supseteq$ & $\bC_+$ & $\supseteq$ & $0$ & $\supseteq$ \\
			\hline
			$\rho_5$ & $\supseteq$ & $(\bC_+ \oplus \bC_0 \oplus \bC_1)/\bC$ & $\supseteq$ & $(\bC_+ \oplus \bC_0 \oplus \bC_1)/\bC$ & $\supseteq$ & $0$ & $\supseteq$ & $0$ & $\supseteq$ \\
			\hline
			$\rho_6$ & $\supseteq$ & $(\bC_+ \oplus \bC_0 \oplus \bC_1)/\bC$ & $\supseteq$ & $\bC_0$ & $\supseteq$ & $0$ & $\supseteq$ & $0$ & $\supseteq$ \\
			\hline
		\end{tabular}
	\end{center}
	Note that the vector bundle $\CH$ (and even the vector bundle $\widetilde{\CH}$ on $\widetilde{X}$) does not split. This can be seen using a criterion from Klyachko, that a vector bundle splits if and only if the vector spaces in the filtrations of all the rays form a distributive lattice, or, equivalently are given by coordinate subspaces (compare \cite[Cor. 2.2.3]{klyachkoOld}).
\end{example}

\bibliographystyle{alpha}
\bibliography{dispExt}

\end{document}